\documentclass[12pt, leqno]{amsart}

\usepackage{amssymb}

\usepackage[all]{xy}

\theoremstyle{plain}
\newtheorem*{thm}{Theorem}
\newtheorem{theorem}{Theorem}[section]
\newtheorem{lemma}[theorem]{Lemma}
\newtheorem{proposition}[theorem]{Proposition}
\newtheorem{corollary}[theorem]{Corollary}

\theoremstyle{definition}

\newtheorem{remark}[theorem]{Remark}

\newtheorem{examples}[theorem]{Examples}

\newcommand\bG{{\mathbb G}}

\newcommand\bQ{{\mathbb Q}}

\newcommand\bZ{{\mathbb Z}}

\newcommand\cA{{\mathcal A}}
\newcommand\cB{{\mathcal B}}
\newcommand\cC{{\mathcal C}}

\newcommand\cE{{\mathcal E}}
\newcommand\cF{{\mathcal F}}

\newcommand\cI{{\mathcal I}}

\newcommand\cL{{\mathcal L}}

\newcommand\cO{{\mathcal O}}
\newcommand\cP{{\mathcal P}}

\newcommand\cV{{\mathcal V}}

\newcommand\tA{{\widetilde A}}

\newcommand\tG{{\widetilde G}}

\newcommand\tP{{\widetilde P}}

\newcommand\tX{{\widetilde X}}
\newcommand\tW{{\widetilde W}}

\newcommand\ucA{\underline{\mathcal A}}

\newcommand\wA{\widehat{A}}

\newcommand\charac{{\rm char}}

\newcommand\cont{{\rm cont}}

\newcommand\id{{\rm id}}

\renewcommand\mod{{\rm mod}}

\newcommand\R{{\rm R}}

\newcommand\Coker{{\rm Coker}}

\newcommand\End{{\rm End}}
\newcommand\Ext{{\rm Ext}}

\newcommand\Gal{{\rm Gal}}

\newcommand\Hom{{\rm Hom}}
\renewcommand\Im{{\rm Im}}

\newcommand\Ker{{\rm Ker}}

\newcommand\Mod{{\rm Mod}}

\newcommand\Pro{{\rm Pro}}

\numberwithin{equation}{section}

\title[Fundamental groups of algebraic groups]
{On the fundamental groups of commutative algebraic groups}

\author{Michel Brion}

\date{}

\begin{document}

\begin{abstract}
Consider the abelian category $\cC$ of commutative group schemes 
of finite type over a field $k$, its full subcategory $\cF$ 
of finite group schemes, and the associated pro-category $\Pro(\cC)$ 
(resp.~$\Pro(\cF)$) of pro-algebraic (resp.~profinite) group schemes. 
When $k$ is perfect, we show that the profinite fundamental group 
$\varpi_1 : \Pro(\cC) \to \Pro(\cF)$ 
is left exact and commutes with base change under algebraic field 
extensions; as a consequence, the higher profinite homotopy functors 
$\varpi_i$ vanish for $i \geq 2$. Along the way, we describe the 
indecomposable projective objects of $\Pro(\cC)$ over an arbitrary 
field $k$.
\end{abstract}

\maketitle

%\tableofcontents

\section{Introduction}
\label{sec:int}

Every real Lie group $G$ gives rise to two exact sequences
\[ 0 \to G^0 \to G \to \pi_0(G) \to 0, \quad 
0 \to \pi_1(G) \to \tG \to G^0 \to 0, \]
where $G^0$ denotes the identity component, $\tG$ its universal 
cover, and $\pi_0(G)$, $\pi_1(G)$ are discrete groups; moreover,
the second homotopy group $\pi_2(G)$ vanishes. 
This classical result has a remarkable analogue for commutative 
algebraic groups over an algebraically closed field $k$, 
as shown by Serre and Oort via a categorical approach 
(see \cite{Se60,Oort}). More specifically, 
consider the category $\cC$ of commutative $k$-group schemes 
of finite type, and the full subcategory $\cF$ of finite group 
schemes; then $\cC$ is an artinian abelian category, and $\cF$ 
is a Serre subcategory. Let $\Pro(\cC)$ (resp.~$\Pro(\cF)$) 
denote the associated pro-category, consisting of pro-algebraic 
(resp.~profinite) group schemes; recall that these categories 
have enough projectives, and $\cC$ (resp.~$\cF$) is equivalent 
to the full subcategory of $\Pro(\cC)$ (resp.~$\Pro(\cF)$) consisting
of artinian objects. Assigning to each object of $\Pro(\cC)$
its largest profinite quotient yields a right exact functor 
\[ \varpi_0  : \Pro(\cC) \longrightarrow  \Pro(\cF). \] 
It turns out that the left derived functors, 
\[ \varpi_i  := L^i \varpi_0: \Pro(\cC) \longrightarrow \Pro(\cF), \] 
vanish for $i \geq 2$; equivalently, $\varpi_1$ is left exact. 
Moreover, $\varpi_0, \varpi_1$ fit in an exact sequence 
\[ 0 \longrightarrow \varpi_1(G) \longrightarrow \tG \longrightarrow
G \longrightarrow \varpi_0(G) \longrightarrow 0 \] 
for any $G \in \Pro(\cC)$ (see \cite[6.2, 10.2]{Se60} 
when $k$ has characteristic $0$, and \cite[II.7, II.14]{Oort} 
in positive characteristics).

The construction of the ``profinite homotopy functors''
$\varpi_i$ makes sense over an arbitrary field $k$; it is easy 
to extend the above exact sequence to this setting. 
The main result of this paper generalizes those of Serre and 
Oort as follows:

\begin{thm}\label{thm:funda}
When $k$ is perfect, the functor
$\varpi_1 : \Pro(\cC) \to \Pro(\cF)$ 
is left exact and commutes with base change under 
algebraic field extensions. As a consequence, the higher 
profinite homotopy functors $\varpi_i$ vanish for $i \geq 2$.
\end{thm}

Our approach is independent of the
general theory of \'etale homotopy groups of schemes 
(see e.g. \cite{AM, Friedlander}). We rather develop an
ad hoc theory of homotopy groups in the setting of pairs
$(\cA,\cB)$, where $\cA$ is an artinian abelian category,
and $\cB$ a Serre subcategory of $\cA$. For this, we
build on constructions and results of Gabriel (see 
\cite[Chap.~III]{Gabriel}) and on further developments
in \cite{Br17b}, recalled in Subsection 
\ref{subsec:pro-artinian}. These may be conveniently
formulated in terms of orthogonal or perpendicular 
categories (see \cite[II.2]{BR} and \cite{GL} for these two
notions). Homotopy groups are introduced in Subsection
\ref{subsec:homotopy}, which generalizes results of 
Demazure and Gabriel on the profinite homotopy groups
of affine group schemes (see \cite[V.3.3]{DG}). Subsection
\ref{subsec:compatibility} investigates compatibility
properties of homotopy groups in the presence of a Serre
subcategory $\cC$ of $\cB$.

In Section \ref{sec:fundamental}, we first apply this
formalism to the category $\cC$ of (commutative)
algebraic groups, and its full subcategory $\cL$ of linear 
algebraic groups, over an arbitrary field $k$; then $\Pro(\cL)$
is equivalent to the category of affine $k$-group schemes,
in view of \cite[V.2.2.2]{DG}. 
The resulting homotopy functor $\pi_1^{\cC,\cL}$ turns out 
to be left exact (Proposition \ref{prop:affine}). 
We then consider the pair $(\cC,\cF)$, 
and obtain the left exactness of $\varpi_1 = \pi_1^{\cC,\cF}$
when $k$ is perfect; in addition, we show that the profinite 
universal cover $\tG$ has homological dimension at most $1$ 
for any $G \in \Pro(\cC)$ (Theorem \ref{thm:perfect}).

When $G$ is an abelian variety over an arbitrary field $k$, 
we construct a minimal projective resolution of $\tG$ 
(Theorem \ref{thm:P(A)}). We also describe the projective objects
of $\Pro(\cC)$ (Proposition \ref{prop:projind}); for this, we use 
results of Demazure and Gabriel on the projectives of $\Pro(\cL)$ 
over a perfect field (see \cite[V.3.7]{DG}), combined with
properties of the isogeny category $\cC/\cF$ (see
\cite{Br17a}). We then show that the profinite homotopy 
functors commute with base change under separable 
algebraic field extensions (Proposition \ref{prop:separable}), 
thereby completing the proof of the main result.

As an application of the above developments, we 
obtain a spectral sequence \`a la Milne (see \cite{Milne}), 
which relates the extension groups in $\cC$ and in the 
corresponding category over a Galois extension of $k$.
Further applications, to the structure of homogeneous
vector bundles over abelian varieties, are presented in
\cite{Br18}.

When the ground field $k$ has characteristic $p > 0$, 
the prime-to-$p$ part $\varpi_1^{(p')}$ of the profinite 
fundamental group commutes with 
arbitrary field extensions, and hence is left exact 
(Proposition \ref{prop:p'}). But over an imperfect field $k$, 
the functors $\varpi_0,\varpi_1$ do not commute with 
purely inseparable field extensions, nor does the pro-\'etale
$p$-primary part of $\varpi_1$ (see Examples 
\ref{ex:inseparable}). In this setting, it seems very likely 
that $\varpi_2$ is nontrivial, but we have no explicit example 
for this; also, the profinite fundamental group scheme
$\varpi_1$ deserves further investigation, already for 
smooth connected unipotent groups.

Finally, it would be interesting to relate the above (affine, 
profinite or pro-\'etale) fundamental groups with further 
notions of fundamental group schemes considered in the
literature. In this direction, note that the
profinite fundamental group of any abelian variety $A$
coincides with Nori's fundamental group scheme (defined
in \cite{No76, No82}), as shown by Nori himself in 
\cite{No83}. Also, when $k$ is algebraically closed, 
the affine fundamental group of $A$ coincides with 
its $S$-fundamental group scheme introduced by Langer 
in \cite{La11}, as follows from \cite[Thm.~6.1]{La12}.

\section{Homotopy groups in pro-artinian categories}
\label{sec:homotopy}

\subsection{Pro-artinian categories and colocalizing subcategories}
\label{subsec:pro-artinian}

Consider an artinian abelian category $\cA$, and the associated 
pro-category $\Pro(\cA)$. Then $\Pro(\cA)$ is a \emph{pro-artinian} 
category in the sense of \cite[V.2.2]{DG}; equivalently, the opposite
category is a Grothendieck category. Moreover, $\cA$ is equivalent to 
the Serre subcategory of $\Pro(\cA)$ consisting of artinian objects 
(see \cite[V.2.3.1]{DG}).
Let $\cB$ be a Serre subcategory of $\cA$; then we may view 
$\Pro(\cB)$ as a Serre subcategory of $\Pro(\cA)$, stable under 
inverse limits (see \cite[Lem.~2.11]{Br17b}). We denote by 
${^{\perp}\Pro(\cB)}$ the full subcategory of $\Pro(\cA)$ 
with objects those $X$ such that
$\Hom_{\Pro(\cA)}(X,Y) = 0$ for all $Y \in \Pro(\cB)$
(this is the left orthogonal subcategory to $\Pro(\cB)$ in 
$\Pro(\cA)$ in the sense of \cite[II.1]{BR}).

\begin{lemma}\label{lem:perp}

Let $X \in \Pro(\cA)$.

\begin{enumerate}

\item[{\rm (i)}] $X \in {^{\perp}\Pro(\cB)}$ if and only if 
$\Hom_{\Pro(\cA)}(X,Y) = 0$ for all $Y \in \cB$.

\item[{\rm (ii)}] $X$ has a smallest subobject $X^{\cB}$ 
in $\Pro(\cA)$ such that
$X/X^{\cB} \in \Pro(\cB)$. Moreover, $X^{\cB} \in {^{\perp}\Pro(\cB)}$.

\item[{\rm (iii)}] For any morphism $f : X \to Y$ in $\Pro(\cA)$, we have
$f(X^{\cB}) \subset Y^{\cB}$ with equality if $f$ is an epimorphism.
If in addition $f$ is essential and $Y \in {^{\perp}\Pro(\cB)}$, then
$X \in {^{\perp}\Pro(\cB)}$. 
\end{enumerate}

\end{lemma}

\begin{proof}
(i) Let $Y \in \Pro(\cB)$. Then $Y = \lim_{\leftarrow} Y_i$, where 
$Y_i \in \cB$. Thus,
$\Hom_{\Pro(\cA)}(X,Y) = \lim_{\leftarrow} \Hom_{\cA}(X,Y_i) = 0$.

(ii) Let $(X_i)_{i \in I}$ be a family of subobjects of $X$ such that
$X/X_i \in \Pro(\cB)$ for all $i$. Then $X/(\cap_{i \in I} X_i)$ is
a subobject of $\prod_{i \in I} X/X_i$, and hence an object of 
$\Pro(\cB)$. This shows the existence of $X^{\cB}$. 

If there exists a nonzero morphism $f : X ^{\cB} \to Y$ for some
$Y \in \Pro(\cB)$, then $X' := \Ker(f)$ is a subobject of $X^{\cB}$
such that $X^{\cB} /X'$ is a nonzero object of $\Pro(\cB)$. It follows
that $X/X' \in \Pro(\cB)$, contradicting the minimality of $X^{\cB}$.
So $X^{\cB} \in {^{\perp}\Pro(\cB)}$.

(iii) The composition $X^{\cB} \to X \to Y \to Y/Y^{\cB}$ is zero,
hence $f(X^{\cB}) \subset Y^{\cB}$. If $f$ is an epimorphism,
then it induces an epimorphism $X/X^{\cB} \to Y/f(X^{\cB})$.
So $Y/f(X^{\cB}) \in \Pro(\cB)$, i.e., $Y^{\cB} \subset f(X^{\cB})$.
Hence $Y^{\cB} = f(X^{\cB})$. If in addition $f$ is essential
and $Y \in {^{\perp}\Pro(\cB)}$, then $Y = f(X^{\cB})$ and hence
$X^{\cB} = X$.
\end{proof} 

In view of Lemma \ref{lem:perp}, every $X \in \Pro(\cA)$ 
lies in a unique exact sequence
\begin{equation}\label{eqn:exact} 
0 \longrightarrow X^{\cB} \longrightarrow X \longrightarrow 
X_{\cB} \longrightarrow 0, 
\end{equation}
where $X^{\cB} \in {^{\perp}\Pro(\cB)}$ and $X_{\cB} \in \Pro(\cB)$. 
Moreover, every $f \in \Hom_{\Pro(\cA)}(X,Y)$ induces 
compatible morphisms 
\[ f^{\cB} : X^{\cB} \longrightarrow Y^{\cB}, \quad 
f_{\cB} : X_{\cB} \longrightarrow Y_{\cB}. \]
This defines a functor 
\[ \pi_0 = \pi_0^{\cA, \cB} : \Pro(\cA) \longrightarrow \Pro(\cB), \quad
X \longmapsto X_{\cB}. \]
Since $\Hom_{\Pro(\cA)}(X^{\cB}, Y) = 0$ for any $Y \in \Pro(\cB)$,
the natural map
\[ \Hom_{\Pro(\cB)}(X_{\cB},Y) \longrightarrow 
\Hom_{\Pro(\cA)}(X,Y) \]
is an isomorphism. In other words, $\pi_0$ is left adjoint
to the inclusion of $\Pro(\cB)$ in $\Pro(\cA)$. As a consequence,
$\pi_0$ is right exact and sends any projective object of
$\Pro(\cA)$ to a projective object of $\Pro(\cB)$.

\begin{lemma}\label{lem:com}
The functor $\pi_0$ commutes with filtered inverse limits.
\end{lemma}

\begin{proof}
Consider a filtered inverse system $(X_i)$ of objects of $\Pro(\cA)$.
This yields a filtered inverse system $(X_i^{\cB})$ of objects
of  ${^{\perp}\Pro(\cB)}$; moreover, we have an isomorphism
\[ \lim_{\to} \Hom_{\Pro(\cA)}(X_i^{\cB},Y) 
\stackrel{\cong}{\longrightarrow} 
\Hom_{\Pro(\cA)}(\lim_{\leftarrow} X_i^{\cB},Y) \]
for any $Y \in \cA$ (see \cite[V.2.3.3]{DG}).
Thus, 
$\Hom_{\Pro(\cA)}(\lim_{\leftarrow} X_i^{\cB},Y) = 0$
for any $Y \in \cB$. In view of Lemma \ref{lem:perp},
it follows that 
$\lim_{\leftarrow} X_i^{\cB} \in {^{\perp}\Pro(\cB)}$.
Also, we have an isomorphism
\[ (\lim_{\leftarrow} X_i)/ (\lim_{\leftarrow} X_i^{\cB})
\cong \lim_{\leftarrow} (X_i)_{\cB} \]
by exactness of inverse limits (see \cite[V.2.2]{DG}).
So
$(\lim_{\leftarrow} X_i)/ (\lim_{\leftarrow} X_i^{\cB})$
is an object of $\Pro(\cB)$; this yields the assertion.
\end{proof}

We denote by 
\[ Q = Q^{\cA,\cB} : \Pro(\cA) \longrightarrow \Pro(\cA)/\Pro(\cB) \]
the quotient functor. Then $Q$ is exact, and commutes with
inverse limits in view of \cite[III.4.Prop.~9]{Gabriel}. Also, 
recall from [loc.~cit., III.4.Prop.~8, Cor.~1] that $Q$ 
has a left adjoint: the \emph{cosection},
\[ C = C^{\cA,\cB} : \Pro(\cA)/\Pro(\cB) \longrightarrow \Pro(\cA), \]
which also commutes with inverse limits and sends projectives
to projectives.
In other words, $\Pro(\cB)$ is a \emph{colocalizing}
subcategory of $\Pro(\cA)$, in the dual sense of 
[loc.~cit., III.2]. Conversely, every colocalizing subcategory
of $\Pro(\cA)$ is equivalent to $\Pro(\cB)$ for a unique
Serre subcategory $\cB$ of $\cA$, in view of 
[loc.~cit., III.4.Prop.~10] and \cite[Rem.~2.13]{Br17b}.
Moreover, $\Pro(\cA)/\Pro(\cB)$ is equivalent to
$\Pro(\cA/\cB)$ by \cite[Prop.~2.12]{Br17b}.

By \cite[III.2.Cor.]{Gabriel}, the essential image of $C$ 
consists of those $X \in \Pro(\cA)$ such that
\begin{equation}\label{eqn:homext}
\Hom_{\Pro(\cA)}(X,Y) = 0 = \Ext^1_{\Pro(\cA)}(X,Y)
\text{ for all } Y \in \Pro(\cB)
\end{equation}
(these are the objects of the left perpendicular subcategory 
to $\Pro(\cB)$ in $\Pro(\cA)$, as defined in \cite{GL}).
Moreover, for any $X \in \Pro(\cA)$, the adjunction map
$CQ(X) \to X$ has its kernel and cokernel in $\Pro(\cB)$
(see [loc.~cit., III.2.Prop.~3]). This yields an exact sequence
in $\Pro(\cA)$
\begin{equation}\label{eqn:univ}
0 \longrightarrow Y_1 
\stackrel{\iota}{\longrightarrow} \tX 
\stackrel{\rho}{\longrightarrow} X 
\stackrel{\gamma}{\longrightarrow} Y_0 
\longrightarrow 0,
\end{equation}
where we set $\tX = \tX^{\cA,\cB} := CQ(X)$ (in particular, 
$\tX \in {^{\perp}\Pro(\cB)}$), and we have $Y_0,Y_1 \in \Pro(\cB)$. 
Note that the long exact sequence (\ref{eqn:univ}) 
depends functorially on $X$. Also, note the natural isomorphism
\[ \Hom_{\Pro(\cA)}(\tX,Y) \cong 
\Hom_{\Pro(\cA)/\Pro(\cB)}(Q(X),Q(Y)) \]
for any $Y \in \Pro(\cA)$. In particular, if $X,Y \in \cA$ then
\begin{equation}\label{eqn:adj}
\Hom_{\Pro(\cA)}(\tX,Y) \cong \Hom_{\cA/\cB}(Q(X),Q(Y)). 
\end{equation}

\begin{lemma}\label{lem:isos}
With the above notation, we have $\rho(\tX) = X^{\cB}$
and the induced epimorphism $\eta: \tX \to X^{\cB}$ is essential. 
Also, there are functorial isomorphisms 
\[ 
\pi_0(X^{\cB}) \stackrel{\cong}{\longrightarrow} Y_0, \quad
\Hom_{\Pro(\cB)}(Y_1,Y) \stackrel{\cong}{\longrightarrow}
\Ext^1_{\Pro(\cA)}(X^{\cB}, Y) \text{ for all } Y \in \Pro(\cB). 
\]
\end{lemma}

\begin{proof}
In view of (\ref{eqn:homext}) and the exact sequence
\[ 0 \longrightarrow Y_1 \longrightarrow \tX \longrightarrow
\rho(\tX) \longrightarrow 0, \]
we obtain the vanishing of $\Hom_{\Pro(\cA)}(\rho(\tX),Y)$ 
and an isomorphism
\[ \Hom_{\Pro(\cB)}(Y_1,Y) \stackrel{\cong}{\longrightarrow}
\Ext^1_{\Pro(\cA)}(\rho(\tX), Y) \] 
for all $Y \in \Pro(\cB)$. Thus, $\rho(\tX) \in {^{\perp}\Pro(\cB)}$. 
Since $X/\rho(\tX) \cong Y_0 \in \Pro(\cB)$, it follows that
$\rho(\tX) = X^{\cB}$.

It remains to show that $\eta : \tX \to X^{\cB}$ is essential. 
Let $Z$ be a subobject of $\tX$ such that the composition
$Z \to \tX \to X^{\cB}$ is an epimorphism. Then 
$\tX = Y_1 + Z$ and hence 
$\tX/Z \cong Y_1/(Y_1 \cap Z) \in \Pro(\cB)$. As 
$\tX \in {^{\perp}\Pro(\cB)}$, it follows that 
$Y_1/(Y_1 \cap Z) = 0$, i.e., $Y_1$ is a subobject
of $Z$. Thus, $Z = \tX$.
\end{proof}

\begin{lemma}\label{lem:sc}
With the notation of the exact sequence (\ref{eqn:univ}), 
the following conditions are equivalent for $X \in \Pro(\cA)$:

\begin{enumerate}

\item[{\rm (i)}] $Y_0 = Y_1 = 0$.

\item[{\rm (ii)}] $\Hom_{\Pro(\cA)}(X,Y) = 0 = \Ext^1_{\Pro(\cA)}(X,Y)$
for all $Y \in \Pro(\cB)$.

\item[{\rm (iii)}] $\Hom_{\Pro(\cA)}(X,Y) = 0 = \Ext^1_{\Pro(\cA)}(X,Y)$
for all $Y \in \cB$.

\end{enumerate}

\end{lemma}

\begin{proof}
The equivalence (i)$\Leftrightarrow$(ii) holds by
\cite[III.2.Cor.]{Gabriel}. As (ii)$\Rightarrow$(iii) is obvious, 
it suffices to show that (iii)$\Rightarrow$(ii).

Let $X \in \Pro(\cA)$ satisfy (iii), and $Y \in \Pro(\cB)$. Then
$\Hom_{\Pro(\cA)}(X,Y) = 0$ by Lemma \ref{lem:perp} (i). 
Consider an essential epimorphism $f: P \to X$, where 
$P \in \Pro(\cA)$ is projective (such a projective cover of $X$ 
exists in view of \cite[II.6.Thm.~2]{Gabriel}).Then 
$P \in {^{\perp}\Pro(\cB)}$ by Lemma \ref{lem:perp} (iii). 
So the exact sequence 
\[ 0 \longrightarrow X' \longrightarrow P 
\stackrel{f}{\longrightarrow} X 
\longrightarrow 0 \]
yields isomorphisms 
\begin{equation}\label{eqn:dec}
\Hom_{\Pro(\cA)}(X',Y) \stackrel{\cong}{\longrightarrow}
\Ext^1_{\Pro(\cA)}(X,Y) 
\end{equation}
for all $Y \in \Pro(\cB)$. In particular, 
$\Hom_{\Pro(\cA)}(X',Y) = 0$ for all $Y \in \cB$. Thus,
$X' \in  {^{\perp}\Pro(\cB)} $ by Lemma \ref{lem:perp} (i). 
Therefore, $\Ext^1_{\Pro(\cA)}(X,Y) = 0$ for all $Y \in \Pro(\cB)$.
\end{proof}

\subsection{Homotopy groups}
\label{subsec:homotopy}

We denote by 
\[ \pi_i = \pi_i^{\cA, \cB} := L^i \pi_0^{\cA,\cB} : 
\Pro(\cA) \longrightarrow \Pro(\cB) \quad (i \geq 0) \] 
the left derived functors of the right exact functor $\pi_0$.
In view of Lemma \ref{lem:com} together with \cite[V.2.3.8]{DG}, 
the \emph{$i$th homotopy functor} $\pi_i$ commutes
with filtered inverse limits for any $i \geq 0$. Also, for any 
exact sequence
\[ 0 \longrightarrow X_1 \longrightarrow X 
\longrightarrow X_2 \longrightarrow 0 \]
in $\Pro(\cA)$, we have an associated homotopy exact sequence
\begin{equation}\label{eqn:homotopy}  
\cdots \to \pi_{i+1}(X_2) \to \pi_i(X_1) \to
\pi_i(X) \to \pi_i(X_2) \to \pi_{i-1}(X_1) \to \cdots 
\end{equation}

\begin{lemma}\label{lem:pro} 
Assume that every projective object of $\Pro(\cB)$ is projective
in $\Pro(\cA)$. Then:

\begin{enumerate}

\item[{\rm (i)}] $\pi_i(Y) = 0$ for all $Y \in \Pro(\cB)$ and $i \geq 1$.

\item[{\rm (ii)}] 
$\pi_i(X^{\cB}) \stackrel{\cong}{\longrightarrow} \pi_i(X)$
for all $X \in \Pro(\cA)$ and $i \geq 1$. 

\end{enumerate}

\end{lemma}

\begin{proof} 
(i) Let $P_{\bullet}$ be a projective resolution of $Y$ in $\Pro(\cB)$. 
Then $\pi_0(P_{\bullet}) = P_{\bullet}$
is still a projective resolution of $Y$ in $\Pro(\cA)$.

(ii) This follows from (i) in view of the exact sequence
(\ref{eqn:exact}).
\end{proof}

\begin{lemma}\label{lem:fund}
With the assumption of Lemma \ref{lem:pro}, there is a
functorial isomorphism $\pi_1(X) \cong Y_1$ for any 
$X \in \Pro(\cA)$.
\end{lemma}

\begin{proof}
The exact sequence (\ref{eqn:exact}) yields an isomorphism
$Q(X^{\cB}) \to Q(X)$ in $\Pro(\cA)/\Pro(\cB)$, and hence 
an isomorphism $CQ(X^{\cB}) \to CQ(X)$ in $\Pro(\cA)$. 
In turn, this yields an isomorphism
$Y_1(X^{\cB}) \to Y_1(X)$ in $\Pro(\cB)$,
where $Y_1(X^{\cB})$ denotes the kernel of the adjunction map
$CQ(X^{\cB}) \to X^{\cB}$, and $Y_1(X)$ is defined similarly. 
Thus, we may assume that 
$X \in  {^{\perp}\Pro(\cB)}$. We then have an exact sequence
\[ 0 \longrightarrow Y_1 \longrightarrow \tX \longrightarrow 
X \longrightarrow 0, \]
which yields an exact sequence
\[ \pi_1(\tX) \longrightarrow \pi_1(X) \longrightarrow
Y_1 \longrightarrow \pi_0(\tX). \]
Moreover, $\pi_0(\tX) = 0$ by Lemma \ref{lem:sc}. So it
suffices to show that $\pi_1(\tX) = 0$. 

As in the proof of Lemma \ref{lem:sc}, consider an exact sequence
\[ 0 \longrightarrow X' \longrightarrow P 
\stackrel{f}{\longrightarrow} \tX \longrightarrow 0, \]
where $P$ is projective and $f$ is essential. We obtain an exact sequence
 \[ \pi_1(P) \longrightarrow \pi_1(\tX) \longrightarrow
\pi_0(X') \longrightarrow \pi_0(P). \]
Moreover, $\pi_0(P) = 0$ by Lemma \ref{lem:perp} (iii), 
and $\pi_1(P) = 0$ by definition.  Thus, 
$\pi_1(\tX) \cong \pi_0(X')$. Also, recall from
(\ref{eqn:homext}) that $\Ext^1_{\Pro(\cA)}(\tX,Y) = 0$ for all
$Y \in \Pro(\cB)$. Using the isomorphism (\ref{eqn:dec}), this yields 
$\Hom_{\Pro(\cA)}(X',Y) = 0$, and hence $\pi_0(X') = 0$. 
Thus, $\pi_1(\tX) = 0$ as desired.
\end{proof}

In view of Lemmas \ref{lem:isos} and \ref{lem:fund}, 
the exact sequence (\ref{eqn:univ}) can be rewritten
in a more suggestive way. Namely, with the assumption
of Lemma \ref{lem:pro}, we have an exact sequence 
for any $X \in \Pro(\cA)$:
\begin{equation}\label{eqn:long}
0 \longrightarrow \pi_1(X)
\stackrel{\iota_X}{\longrightarrow} \tX
\stackrel{\rho_X}{\longrightarrow} X
\stackrel{\gamma_X}{\longrightarrow} \pi_0(X)
\longrightarrow 0. 
\end{equation}
In particular, when $X \in {^{\perp}\Pro(\cB)}$,
we obtain an extension
\begin{equation}\label{eqn:funda}
0 \longrightarrow \pi_1(X)
\longrightarrow \tX \longrightarrow X
\longrightarrow 0.
\end{equation}

Using Lemmas \ref{lem:isos} and \ref{lem:fund} again, 
this yields in turn:

\begin{corollary}\label{cor:sc}
With the assumption of Lemma \ref{lem:pro}, let 
$X \in {^{\perp}\Pro(\cB)}$ and $Y \in \Pro(\cB)$. Then
$\Hom_{\Pro(\cA)}(\pi_1(X),Y) 
\stackrel{\cong}{\to} \Ext^1_{\Pro(\cA)}(X,Y)$
via pushout by the extension (\ref{eqn:funda}).
\end{corollary}

In other words, (\ref{eqn:funda}) is the universal 
extension of $X$ by an object of $\Pro(\cB)$. We now 
record a similar uniqueness result for the exact
sequence (\ref{eqn:long}), to be used in Subsection
\ref{subsec:profinite}.

\begin{lemma}\label{lem:unique}
With the assumption of Lemma \ref{lem:pro},
consider an exact sequence
\begin{equation}\label{eqn:unique}
0 \longrightarrow Y_1 \longrightarrow X' \longrightarrow 
X \longrightarrow Y_0 \longrightarrow 0
\end{equation} 
in $\Pro(\cA)$, where $Y_0,Y_1 \in \Pro(\cB)$
and $X'$ is in the essential image of $C$.
Then there is a commutative diagram of exact sequences
\[ \xymatrix{
0 \ar[r] & \pi_1(X) \ar[r] \ar[d]_{f_1} & \tX \ar[r] \ar[d]_{f'} & 
X  \ar[r] \ar[d]_{f} & \pi_0(X)  \ar[r] \ar[d]_{f_0} & 0 \\
0 \ar[r] & Y_1 \ar[r]  & X'  \ar[r] & X  \ar[r]  & Y_0 \ar[r]  & 0, \\
} \]
where $f_1,f',f,f_0$ are isomorphisms.
\end{lemma}

\begin{proof}
Cut the exact sequence (\ref{eqn:unique}) 
in two short exact sequences
\[ 0 \longrightarrow Y_1 \longrightarrow X' 
\longrightarrow X''  \longrightarrow 0, \quad 
0 \longrightarrow X'' \longrightarrow X 
\longrightarrow Y_0  \longrightarrow 0. \]
Since $X' $ is an object of ${^{\perp}\Pro(\cB)}$, 
so is $X''$. As $Y_0 \in \cB$, we obtain 
a commutative diagram of exact sequences
\[ \xymatrix{
0 \ar[r] &  X^{\cB} \ar[r] \ar[d]_{f''} & 
X  \ar[r] \ar[d]_{f} & \pi_0(X)  \ar[r] \ar[d]_{f_0} & 0 \\
0 \ar[r] & X''  \ar[r] & X  \ar[r]  & Y_0 \ar[r]  & 0, \\
} \]
where the vertical arrows are isomorphisms. 
As a consequence, we may replace $X$ with 
$X^{\cB}$, and assume that $\pi_0(X) = 0 = Y_0$.

Also, the induced morphism $Q(X') \to Q(X)$
is an isomorphism, and hence so is 
$CQ(X') \to CQ(X) = \tX$.  Since the adjunction
$CQ(X') \to X'$ is an isomorphism, this yields
an isomorphism $\tX \cong X'$. Thus, we may 
further assume that (\ref{eqn:unique}) is of the form
\[ 0 \longrightarrow Y_1 \longrightarrow \tX
\longrightarrow X \longrightarrow 0. \]
Then the associated map
$\Hom_{\Pro(\cB)}(Y_1,Y) \to \Ext^1_{\Pro(\cA)}(X,Y)$
is an isomorphism for all $Y \in \Pro(\cB)$, 
by Lemma \ref{lem:sc}. In view of the uniqueness 
of the universal extension of $X$ by an object of 
$\Pro(\cB)$, this completes the proof.
\end{proof}

Next, we obtain two reformulations of the left exactness
of the functor $\pi_1$:

\begin{lemma}\label{lem:vanishing}
With the assumption of Lemma \ref{lem:pro}, the following
conditions are equivalent:

\begin{enumerate}

\item[{\rm (i)}] The cosection functor 
$C : \Pro(\cA)/\Pro(\cB) \to \Pro(\cA)$ is exact. 

\item[{\rm (ii)}] $\pi_1$ is left exact. 

\item[{\rm (iii)}] $\pi_i = 0$ for all $i \geq 2$.

\end{enumerate}

\end{lemma}

\begin{proof}
(i)$\Rightarrow$(ii) 
Consider an exact sequence 
\[ 0 \longrightarrow  X_1 \longrightarrow X 
\longrightarrow X_2 \longrightarrow 0 \] 
in $\Pro(\cA)$. Then we have a commutative diagram 
of exact sequences
\[ \xymatrix{
0 \ar[r] & \tX_1 \ar[r] \ar[d] & \tX \ar[r] \ar[d] & \tX_2 \ar[r] \ar[d] & 0 \\
0 \ar[r] & X_1 \ar[r]  & X \ar[r] & X_2 \ar[r]  & 0. \\
} \]
In view of the exact sequence (\ref{eqn:long}) and its analogues for
$X_1$, $X_2$,  the snake lemma yields an exact sequence
\[ 0 \to \pi_1(X_1) \to \pi_1(X)
\to \pi_1(X_2) \to \pi_0(X_1)
\to \pi_0(X) \to \pi_0(X_2) \to 0. \]
In particular, $\pi_1$ is left exact.

(ii)$\Rightarrow$(i)
This follows from the dual statement of \cite[III.3.Prop.~7]{Gabriel}. 

(ii)$\Rightarrow$(iii) 
This is obtained by a standard argument that we recall for 
completeness. Let $X \in \Pro(\cA)$ and choose a projective cover 
\[ 0 \longrightarrow X' \longrightarrow P 
\longrightarrow X \longrightarrow 0. \] 
As $\pi_i(P) = 0$ for all $i \geq 1$, we obtain isomorphisms
$\pi_i(X) \stackrel{\cong}{\to} \pi_{i - 1}(X')$
for all $i \geq 2$. Since $X'$ is a subobject of $P$, 
we have $\pi_1(X') = 0$ by left exactness, hence
$\pi_2(X) = 0$. Iterating this argument
completes the proof. 

(iii)$\Rightarrow$(ii) This follows from the homotopy
exact sequence (\ref{eqn:homotopy}).

\end{proof}

Finally, we record an easy and useful divisibility 
property of homotopy groups. 
For any $X \in \Pro(\cA)$ and any integer
$n$, we denote by $n_X \in \End_{\cA}(X)$ the
multiplication by $n$, and by $X[n]$ its kernel.
We say that $X$ is \emph{divisible} 
(resp.~\emph{uniquely divisible}) 
if $n_X$ is an epimorphism (resp.~an isomorphism)
for any $n \geq 1$.

\begin{lemma}\label{lem:tors}
With the assumption of Lemma \ref{lem:pro}, 
let $X$ be an object of $\Pro(\cA)$. Assume that $X$ 
is divisible and $X[n] \in \Pro(\cB)$ for any $n \geq 1$
(in particular, $\pi_i(X[n]) = 0$ for any such $n$ 
and any $i \geq 1$). Then $\tX$ and the $\pi_i(X)$  
($i \geq 2$) are uniquely divisible. Moreover, there is 
an exact sequence
\[ 0 \longrightarrow \pi_1(X) 
\stackrel{n}{\longrightarrow} \pi_1(X)
\longrightarrow X[n] \longrightarrow \pi_0(X) 
\stackrel{n}{\longrightarrow} \pi_0(X)
\longrightarrow 0 \]
for any $n \geq 1$.
\end{lemma}

\begin{proof}
By assumption, we have an exact sequence
\begin{equation}\label{eqn:divisible}
0 \longrightarrow X[n] \longrightarrow X
\stackrel{n_X}{\longrightarrow} X \longrightarrow 0 
\end{equation}
for any $n \geq 1$. Thus, $n_X$ induces 
an automorphism of $Q(X)$, and hence of 
$CQ(X) = \tX$. In other words, $\tX$ is uniquely divisible.
The remaining assertions follow from the homotopy 
exact sequence associated with (\ref{eqn:divisible}).
\end{proof}

\subsection{Structure of projective objects}
\label{subsec:proj}

In this subsection, we consider an artinian abelian category
$\cA$ and a Serre subcategory $\cB$ such that every
projective object of $\Pro(\cB)$ is projective in $\Pro(\cA)$.
Our aim is to describe the projectives of $\Pro(\cA)$
in terms of those of $\Pro(\cB)$ and 
$\Pro(\cA)/\Pro(\cB) \cong \Pro(\cA/\cB)$. 
We first obtain a generalization of \cite[V.3.3.9]{DG}:

\begin{lemma}\label{lem:proj}
For any projective object $P \in \Pro(\cA)$, there is an isomorphism
$P \cong P^{\cB} \oplus \pi_0(P)$ which is compatible
with $\gamma_P : P \to \pi_0(P)$. Moreover,
$\tP \cong P^{\cB}$.
\end{lemma}

\begin{proof}
Recall that $\pi_0$ is left adjoint to the inclusion
of $\Pro(\cB)$ in $\Pro(\cA)$. It follows
that $\pi_0(P)$ is projective in $\Pro(\cB)$, 
and hence in $\Pro(\cA)$ as well. This yields a compatible 
isomorphism $P \cong P^{\cB} \oplus \pi_0(P)$.
In particular, $P^{\cB}$ is projective, and hence in the 
essential image of $C$ by (\ref{eqn:homext}). So the
adjunction map $CQ(P^{\cB}) \to P^{\cB}$ is an 
isomorphism. As 
$CQ(P^{\cB}) \stackrel{\cong}{\to} CQ(P) = \tP$,
this completes the proof.
\end{proof}

\begin{corollary}\label{cor:lifting}
Let $f: X \to Y$ be an epimorphism in $\Pro(\cA)$, 
where $Y$ is an object of $\Pro(\cB)$. Then there exists 
a subobject $Y'$ of $X$ such that $Y' \in \Pro(\cB)$ 
and the composition $Y' \to X \to Y$ is an epimorphism.
\end{corollary}

\begin{proof}
We may assume that $X$ is projective. By Lemma
\ref{lem:proj}, we may then choose
an isomorphism $X \cong \tX \oplus \pi_0(X)$
compatibly with $\gamma_X : X \to \pi_0(X)$.
Since $\pi_0(f) : \pi_0(X) \to \pi_0(Y)$ 
is an epimorphism, and $\gamma_Y : Y \to \pi_0(Y)$ 
is an isomorphism, the statement holds with
$Y' = \pi_0(X)$.
\end{proof}

The above corollary asserts that the pair 
$(\Pro(\cA),\Pro(\cB))$ satisfies the \emph{lifting property}
introduced in \cite[\S 2.2]{Br17b}. Thus, this property holds 
for the pair $(\cA,\cB)$ as well. Conversely, if $(\cA,\cB)$
satisfies the lifting property, then every projective object
in $\Pro(\cB)$ is projective in $\Pro(\cA)$ by
\cite[Lem.~2.14]{Br17b}.

Next, recall from \cite[V.2.4]{DG} that every projective object
of $\Pro(\cA)$ is a product of indecomposable projectives, 
unique up to reordering; moreover, the indecomposable
projectives are projective covers of objects of $\cA$. Also,
given $X \in \Pro(\cA)$ such that $Q(X)$ is projective
in $\Pro(\cA/\cB)$, the adjunction map $\rho: \tX = CQ(X) \to X$ 
is the projective cover of $X$ (indeed, $C$ sends projectives 
to projectives, and $\rho$ is essential by Lemma 
\ref{lem:isos}). Together with Lemma \ref{lem:proj}, this yields
the following result (see also \cite[III.3.Cor.~2]{Gabriel}):

\begin{corollary}\label{cor:indproj} 
The indecomposable projectives of $\Pro(\cA)$ are exactly
those of $\Pro(\cB)$ and the $\tX$, where 
$X \in {^{\perp}\Pro(\cB)}$ and $Q(X)$ is indecomposable
projective in $\Pro(\cA/\cB)$. 
\end{corollary}

The latter indecomposable projectives can be constructed 
as follows: 

\begin{lemma}\label{lem:ess}
Let $X \in {^{\perp}\Pro(\cB)}$.

\begin{enumerate}

\item[{\rm (i)}] Consider an exact sequence in $\Pro(\cA)$,
\[ 0 \longrightarrow Z \longrightarrow Y
\stackrel{f}{\longrightarrow} X \longrightarrow 0. \] 
Then $f$ is essential if and only if $Z \in \Pro(\cB)$ and 
$Y \in {^{\perp} \Pro(\cB)}$.

\item[{\rm (ii)}] Assume that $Q(X)$ is projective in 
$\Pro(\cA)/\Pro(\cB)$. Then the essential epimorphisms
$f : Y \to X$, where $\Ker(f) \in \cB$, 
form a filtered inverse system with limit 
the projective cover of $X$ in $\Pro(\cA)$.

\end{enumerate}

\end{lemma}

\begin{proof}
(i) Note that $f$ induces an epimorphism
$Y/Y^{\cB} \to X/f(X^{\cB})$. Since 
$Y/Y^{\cB} \in \Pro(\cB)$ and
 $X/f(X^{\cB})\in {^{\perp} \Pro(\cB)}$ ,
we must have $X/f(X^{\cB}) = 0$, i.e., the
composition $Y^{\cB} \to Y \to X$ is an
epimorphism.

Assume that $f$ is essential. Then $Y^{\cB} = Y$,
i.e., $Y \in {^{\perp} \Pro(\cB)}$. Also, by the
lifting property, we have $Z = Z^{\cB} + W$
for some subobject $W$ of $Z$ such that 
$W \in \Pro(\cB)$. This yields an exact sequence
\[ 0 \longrightarrow Z/W \longrightarrow Y/W 
\longrightarrow X/f(W) \longrightarrow 0 \]
in $\Pro(\cA)$, and hence in $\Pro(\cA/\cB)$. 
As $X/f(W) \cong X$ is projective in the latter category,
this sequence is split by some $g \in \Hom_{\Pro(\cA/\cB)}(X,Y/W)$.
Since $X \in {^{\perp} \Pro(\cB)}$, we may
represent $g$ by $h \in \Hom_{\Pro(\cA)}(X,Y/W')$
for some $W' \subset Y$ such that $W \subset W'$ and
$W' \in \Pro(\cB)$. 
Denote by $p$ the composition of morphisms in $\Pro(\cA)$
\[ X \stackrel{h}{\longrightarrow} Y/W' \longrightarrow X/f(W') \]
(where the morphism on the right is induced by $f$),
and by $q : X \to X/f(W')$ the quotient morphism in $\Pro(\cA)$. 
Then $p$ represents the identity endomorphism of $X$ in 
$\Pro(\cA)/\Pro(\cB)$; thus, $p - q$ represents 
zero there. Using again the assumption that 
$X \in {^{\perp} \Pro(\cB)}$, it follows that $p - q$ is zero
in $\Pro(\cA)$. In particular, the composition $h(X) \to Y/W' \to X$ is
an epimorphism. Since $f$ is essential, $h$ must be
an epimorphism as well. So $g$ is an isomorphism in 
$\Pro(\cA)/\Pro(\cB)$, hence $Z/W \in \Pro(\cB)$. We conclude
that $Z \in \Pro(\cB)$.

Conversely, assume that $Z \in \Pro(\cB)$ and 
$Y \in {^{\perp} \Pro(\cB)}$. Let $Y' \subset Y$ 
such that the composition $Y' \to Y \to X$ is
an epimorphism. Then $Y = Y' + Z$, hence
$Z \to Y \to Y/Y'$ is an epimorphism as well. 
So $Y/Y'$ is an object of $\Pro(\cB)$, and hence 
is zero. We conclude that $f$ is essential.

(ii) Consider two exact sequences
\[ 0 \longrightarrow Z_i \longrightarrow Y_i
\stackrel{f_i}{\longrightarrow}
X \longrightarrow 0 \quad (i = 1,2), \]
where $f_1,f_2$ are essential and $Z_1,Z_2 \in \cB$. 
Then the induced morphism
\[ f : Y_1 \times_X Y_2  =: Y \longrightarrow X \] 
is an epimorphism with kernel $Z_1 \times Z_2$.
In view of (i), it follows that the composition
$Y^{\cB} \to Y \to X$ is an essential epimorphism. Thus, 
these essential epimorphisms form a
filtered inverse system.

Given such an essential epimorphism
$f : Y \to X$, the map $\rho : \tX \to X$ lifts to a morphism
$\varphi_Y :  \tX \to Y$. Moreover,
$\varphi_Y$ is unique (since $\Ker(f) \in \Pro(\cB)$
and $\tX \in {^{\perp} \Pro(\cB)}$), and is an epimorphism 
as well. So we obtain an epimorphism 
\[ \varphi : \tX \longrightarrow 
\lim_{\leftarrow} Y \]
with an obvious notation. To show that $\varphi$ is
a monomorphism, consider the family $(K_i)$
of subobjects of $\Ker(\rho)$ such that 
$\Ker(\rho)/K_i \in \cB$. Then $\tX/K_i \in \cA$ and
$\rho$ factors through an essential epimorphism
$\tX/K_i \to X$; the corresponding morphism
$\varphi_i : \tX \to \tX/K_i$ is just the quotient
morphism. Since $\cap K_i$ is zero, this
completes the proof. 
\end{proof}

\subsection{Compatibility properties}
\label{subsec:compatibility}

Throughout this subsection, we consider an artinian 
abelian category $\cA$, a Serre subcategory $\cB$ 
such that the pair $(\cA,\cB)$ satisfies the lifting 
property, and in addition a Serre subcategory 
$\cC$ of $\cB$. We first relate the homotopy functors
associated to the three pairs $(\cA,\cB)$, $(\cB,\cC)$
and $(\cA,\cC)$:

\begin{lemma}\label{lem:comp}
Let $X \in \Pro(\cA)$.

\begin{enumerate}

\item[{\rm (i)}]
There is a natural isomorphism
$\pi_0^{\cA,\cC}(X) \stackrel{\cong}{\to}
\pi_0^{\cB,\cC} (\pi_0^{\cA,\cB}(X))$. 

\item[{\rm (ii)}]
There is a spectral sequence
$\pi_i^{\cB,\cC} (\pi_j^{\cA,\cB}(X)) \Rightarrow
\pi_{i+j}^{\cA,\cC}(X)$. 

\end{enumerate}

\end{lemma}

\begin{proof}
(i) This follows readily from the definitions.

(ii) Recall that $\pi_0^{\cA,\cB} : \Pro(\cA) \to \Pro(\cB)$
sends projectives to projectives; also, every projective
in $\Pro(\cB)$ is obviously acyclic for $\pi_0^{\cB,\cC}$.
In view of (i), this yields a Grothendieck spectral sequence
as stated.
\end{proof}

\begin{remark}\label{rem:comp}
When $X \in \cB$, the above spectral sequence yields 
isomorphisms 
$\pi_i^{\cB,\cC}(X) \stackrel{\cong}{\to} \pi_i^{\cA,\cC}(X)$
for all $i \geq 0$, in view of Lemma \ref{lem:pro}. 
Alternatively, these isomorphisms follow from the obvious
equality $\pi_0^{\cB,\cC}(X) = \pi_0^{\cA,\cC}(X)$, 
since evey projective object of $\Pro(\cB)$ is projective
in $\Pro(\cA)$.

On the other hand, when $X \in {^{\perp}\Pro(\cB)}$, 
the first terms of the spectral sequence 
yield a natural isomorphism
\[ \pi_1^{\cA,\cC}(X) \stackrel{\cong}{\longrightarrow} 
\pi_0^{\cB,\cC}(\pi_1^{\cA,\cB}(X)). \]
This can also be seen directly: consider the universal 
extension of $X$ by an object of $\Pro(\cB)$,
\[ 0 \longrightarrow Y \longrightarrow \tX 
\longrightarrow X \longrightarrow 0, \]
where $Y := \pi_1^{\cA,\cB}(X)$.
Then one may readily check that the induced 
exact sequence
\[ 0 \longrightarrow Y/Y^{\cC}
\longrightarrow \tX/Y^{\cC} \longrightarrow X
\longrightarrow 0 \]
is the universal extension of $X$ by an object
of $\Pro(\cC)$; thus, 
$Y/Y^{\cC} \cong \pi_1^{\cA,\cB}(X)$. But also
$Y/Y^{\cC} = \pi_0^{\cB,\cC}(\pi_1^{\cA,\cB}(X))$. 
\end{remark}

Next, we investigate the behavior of the homotopy groups 
$\pi_i^{\cA, \cB}$ under the quotient functor
\[ Q^{\cA,\cC} : \Pro(\cA) \longrightarrow \Pro(\cA)/\Pro(\cC). \]
We will need the following observation:

\begin{lemma}\label{lem:standard}
Assume that the pair $(\cA,\cC)$ satisfies the
lifting property. Then $\cB/\cC$ is a Serre
subcategory of $\cA/\cC$, and the quotient
$(\cA/\cC)/(\cB /\cC)$ is naturally equivalent to 
$\cA/\cC$. Moreover, the pair $(\cA/\cC,\cB/\cC)$ 
satisfies the lifting property.
\end{lemma}

\begin{proof}
Let $X \in \cB$, $Y \in \cA$, and let $\varphi : X \to Y$
be an isomorphism in $\cA/\cC$. By \cite[Lem.~2.7]{Br17b}, 
there exists a subobject $Y' \subset Y$ in $\cA$
such that $Y' \in \cC$ and $\varphi$ is represented by 
a morphism $f: X \to Y/Y'$ in  $\cA$. Then $\Ker(f)$ and 
$\Coker(f)$ are objects of $\cC$ in view of 
\cite[III.1.Lem.~2]{Gabriel}. Since $\cB$ is a Serre subcategory 
of $\cA$ containing $\cC$, it follows that $Y \in \cB$. Thus,
 $\cB/\cC$ is a strict subcategory of $\cA/\cC$.

Next, let $0 \to X_1 \to X \to X_2 \to 0$ be an exact 
sequence in $\cA/\cC$. Then there exists 
a commutative diagram in that category
\[ \xymatrix{
0 \ar[r] & X_1 \ar[r] \ar[d]
& X \ar[r] \ar[d] & X_2 \ar[r] \ar[d] & 0 \\
0 \ar[r] & Y_1 \ar[r]  & Y \ar[r] & Y_2 \ar[r]  & 0, \\
} \]
where the vertical arrows are isomorphisms, and the
bottom sequence is the image of an exact sequence
in $\cA$ under the quotient functor $Q^{\cA,\cC}$
(see \cite[Lem.~2.9]{Br17b}). As a consequence,
$X \in \cB$ if and only if $X_1,X_2 \in \cB$. So
$\cB/\cC$ is a Serre subcategory of $\cA/\cC$.

The equivalence of categories
$(\cA/\cC)/(\cB/\cC) \cong \cA/\cB$
follows from the universal property of quotient functors.

We now check that $(\cA/\cC,\cB/\cC)$ satisfies
the lifting property. Let $\varphi : X \to Y$ be 
an epimorphism in $\cA/\cC$. In view of 
\cite[Lem.~2.7]{Br17b} again, replacing $Y$ with 
an isomorphic object in $\cA/\cC$, we may assume 
that $\varphi$ is represented by a morphism 
$f: X \to Y$ in $\cA$; then $\Coker(f)$ is an object 
of $\cC$ by \cite[III.1.Lem.~2]{Gabriel} again.
Next, we may replace $X,Y$ with $X^{\cC}, Y^{\cC}$,
and hence assume that $f$ is an epimorphism in 
$\cA$. Then there exists a subobject $Y'$ of $X$
such that $Y' \in \cB$ and the composition
$Y' \to X \to Y$ is an epimorphism in $\cA$,
hence in $\cA/\cC$.
\end{proof}

\begin{lemma}\label{lem:compatible}
With the assumption of Lemma \ref{lem:standard}, 
$\pi_i^{\cA/\cC,\cB/\cC}(X)$ is naturally isomorphic 
to the image of $\pi_i^{\cA, \cB}(X)$ in $\Pro(\cB/\cC)$, 
for any $i \geq 0$ and any object $X$ of $\Pro(\cA)$.
\end{lemma}

\begin{proof}
Recall that every projective object in 
$\Pro(\cC)$ is projective in $\Pro(\cA)$. By
the dual statement of \cite[III.3.Cor.~3]{Gabriel}, 
it follows that the quotient functor $Q^{\cA,\cC}$ sends 
projectives to projectives. Thus, it suffices to 
check the assertion for $i = 0$.

Let $X \in \Pro(\cA)$ and consider the exact sequence
\[ 0 \longrightarrow X^{\cB} \longrightarrow X
\longrightarrow \pi_0^{\cA,\cB}(X) \longrightarrow 0 \]
in $\Pro(\cA)$, where $X^{\cB} \in {^{\perp}\Pro(\cB)}$.
This sequence is still exact in $\Pro(\cA/\cC)$; thus,
it suffices to show that 
$X^{\cB} \in {^{\perp}\Pro(\cB/\cC)}$. In view of
Lemma \ref{lem:perp}, it suffices in turn to show that 
every morphism $\varphi : X^{\cB} \to Y$ in 
$\Pro(\cA/\cC)$, where $Y \in \cB$, is zero.

In $\Pro(\cA)$, we have 
$X^{\cB} = \lim_{\leftarrow}  X_i$, where $X_i \in \cA$
and the projections $X^{\cB} \to X_i$ are epimorphisms.
Hence this also holds in $\Pro(\cA/\cC)$. Since
\[ \Hom_{\Pro(\cA/\cC)}( \lim_{\leftarrow}  X_i, Y)
= \lim_{\to} \Hom_{\cA/\cC}(X_i, Y), \]
we see that $\varphi$ is represented by a morphism
$\varphi_i : X_i \to Y$ in $\cA/\cC$. Using 
\cite[Lem.~2.7]{Br17b}, it follows that $\varphi$ is represented
by a morphism $f_i : X_i \to Y/Y'$  in $\cA$, for some
$Y' \subset Y$ such that $Y' \in \cC$. The composition
$X^{\cB} \to X_i \to Y/Y'$ is zero, since $Y/Y' \in \cB$.
So $f_i = 0$, and $\varphi = 0$.
\end{proof}

\section{Fundamental groups of commutative algebraic groups}
\label{sec:fundamental}

\subsection{The affine fundamental group}
\label{subsec:affine}

Let $k$ be a field. As in the introduction, we consider 
the artinian abelian category $\cC$ of commutative 
$k$-group schemes of finite type, and the associated 
pro-category $\Pro(\cC)$ of pro-algebraic groups. 
We denote by $\cL$ the full subcategory of $\cC$ 
with objects the affine (or equivalently, linear) algebraic 
groups. Then $\cL$ is a Serre subcategory of $\cC$, 
as follows from fpqc descent (see e.g. \cite[34.20.18]{SP}). 
Also, recall that the pro-category $\Pro(\cL)$ is equivalent 
to the category of commutative affine $k$-group schemes.

By the results of Subsection \ref{subsec:pro-artinian}, 
every object of $\Pro(\cC)$ has a largest affine quotient; 
this yields a right exact functor
\[ \pi_0^{\cC,\cL} : \Pro(\cC) \longrightarrow \Pro(\cL), \]
which commutes with filtered inverse limits and
extends the affinization functor $\cC \to \cL$
considered for example in \cite[III.3.8]{DG}. 
The results of Subsection \ref{subsec:homotopy} also
apply to this setting, in view of the following observation:

\begin{lemma}\label{lem:CL}
The pair $(\cC,\cL)$ satisfies the lifting property.
\end{lemma}

\begin{proof}
Let $G \in \cC$. By a variant of Chevalley's structure 
theorem for algebraic groups (see \cite[Thm.~2.3]{Br17a})
that we will use repeatedly, there is an exact sequence
\begin{equation}\label{eqn:LGA} 
0 \longrightarrow L \longrightarrow G \longrightarrow A
\longrightarrow 0, 
\end{equation}
where $L$ is linear and $A$ is an abelian variety.
Let $f: G \to H$ be an epimorphism, where
$H$ is linear. Then $G' := G/(\Ker(f) + L)$
is linear (as a quotient of $G/\Ker(f) \cong H$)
and is an abelian variety (as a quotient of $G/L \cong A$).
Thus, $G' = 0$, i.e., $G = \Ker(f) + L$.
So the composition $L \to G \to H$ is an epimorphism.
\end{proof}

We now describe the quotient categories $\cC/\cL$
and $\Pro(\cC)/\Pro(\cL)$.
Consider the full subcategory $\cA$ of $\cC$ with
objects the abelian varieties; then $\cA$ is an additive
subcategory, but not a Serre subcategory. Denote by
$\ucA$ the corresponding isogeny category: the objects
of $\ucA$ are those of $\cA$, and the morphisms
are defined by
$\Hom_{\ucA}(G,H) := \Hom_{\cA}(G,H) \otimes_{\bZ} \bQ$.
Then $\ucA$ is a semi-simple artinian abelian category;
its simple objects are exactly the simple abelian varieties,
i.e., those having no non-trivial abelian subvariety.

\begin{lemma}\label{lem:isogeny}
With the above notation, the composite functor
$\cA \to \cC \to \cC/\cL$ induces equivalences
of categories 
\[ \ucA \stackrel{\cong}{\longrightarrow} \cC/\cL,
\quad 
\Pro(\ucA) \stackrel{\cong}{\longrightarrow} 
\Pro(\cC)/\Pro(\cL). \]
Moreover, $\Pro(\ucA)$ is semi-simple.
\end{lemma}

\begin{proof}
Denote by $F : \cA \to \cC/\cL$ the composite functor.
Then $F$ is essentially surjective by Chevalley's 
theorem again. Also, recall from \cite[III.1]{Gabriel} that  
\[ \Hom_{\cC/\cL}(G,H) = \lim_{\to} \Hom_{\cC}(G',H/H') \]
for all $G,H \in \cC$, where $G'$ (resp.~$H'$) runs over the 
subgroup schemes of $G$ such that $G/G'$ is linear
(resp.~the linear subgroup schemes of $H$). 
When $G$ and $H$ are abelian varieties, we must
have $G' = G$; moreover, $H'$ is finite, or equivalently,
contained in the $n$-torsion subgroup scheme $H[n]$ 
for some $n \geq 1$. As a consequence,
\[ \Hom_{\cC/\cL}(G,H) = \lim_{\to} \Hom_{\cC}(G,H/H[n]), \]
where the direct limit is over the positive integers ordered
by divisibility. This yields a natural isomorphism
\[ \Hom_{\cC/\cL}(G,H) \stackrel{\cong}{\longrightarrow}
\Hom_{\cC}(G,H) \otimes_{\bZ} \bQ \]
(see e.g. \cite[Prop.~3.6]{Br17a} for details), and hence 
the first equivalence of categories, $\ucA \cong \cC/\cL$. 
Since $\Pro(\cC/\cL) \cong \Pro(\cC)/\Pro(\cL)$, 
it follows that $\Pro(\ucA) \cong \Pro(\cC)/\Pro(\cL)$.

To show that $\Pro(\ucA)$ is semi-simple, it suffices
to check that every object is projective. In view of
\cite[V.2.3.5]{DG}, it suffices in turn to check that
for any $G \in \Pro(\ucA)$ and any epimorphism
$f: G_1 \to G_2$ in $\ucA$, the induced map
$\Hom_{\ucA}(G,G_1) \to \Hom_{\ucA}(G,G_2)$
is surjective. But this follows from the existence 
of a section of $f$.
\end{proof}

Before stating our next result, we introduce 
some notation. We denote by
\[ Q = Q^{\cC,\cL} : 
\Pro(\cC) \longrightarrow \Pro(\cC)/\Pro(\cL) \]
the quotient functor, and by
\[ C = C^{\cC,\cL} :
\Pro(\cC)/\Pro(\cL) \longrightarrow \Pro(\cC) \]
the associated cosection functor. For any abelian
variety $A$, we set 
\[ P(A) := CQ(A). \]

\begin{proposition}\label{prop:affine}

\begin{enumerate}

\item[{\rm (i)}] The functor $C$ is exact.

\item[{\rm (ii)}] The projective objects of $\Pro(\cC)$ 
are exactly the products of those of  $\Pro(\cL)$ with
the $P(A)$, where $A$ is an abelian variety.
Moreover, $P(A)$ is a projective cover of $A$ 
in $\Pro(\cC)$, and is uniquely divisible.

\end{enumerate}

\end{proposition}

\begin{proof}
(i) Recall that $C$ commutes with inverse limits,
and hence with products. Since the category 
$\Pro(\cC)/\Pro(\cL)$ is semi-simple (Lemma
\ref{lem:isogeny}), this yields the assertion.

(ii) By Lemma \ref{lem:CL} and \cite[Lem.~2.14]{Br17b}, 
every projective object of $\Pro(\cL)$ is projective
in $\Pro(\cC)$. In view of the dual statement of 
\cite[III.3.Cor. 2]{Gabriel}, it follows that the projective
objects of $\Pro(\cC)$ are exactly the products of
those of $\Pro(\cL)$ with the images under $C$ 
of projective objects of $\Pro(\cC)/\Pro(\cL)$. 
Using again the fact that this quotient category
is semi-simple, this yields the first assertion.

Let $A$ be an abelian variety. Since every 
affine quotient of $A$ is trivial, the adjunction map 
$\rho : P(A) \to A$ is an epimorphism. Also, $\rho$
is essential by Lemma \ref{lem:isos}; thus, $P(A)$
is a projective cover of $A$ in $\Pro(\cC)$.
The unique divisibility assertion follows from Lemma
\ref{lem:tors}, since $A$ is divisible and its $n$-torsion
subgroup schemes are finite for all $n \geq 1$.
\end{proof}

\subsection{The profinite fundamental group}
\label{subsec:profinite}

We now consider the Serre subcategory $\cF$ of $\cL$
with objects the finite group schemes.
As in the introduction, we denote by 
\[ \varpi_i := \pi_i^{\cC,\cF} : \Pro(\cC) \longrightarrow \Pro(\cF) \] 
the profinite homotopy functors. For any $G \in \Pro(\cC)$,
the exact sequence (\ref{eqn:long}) may be rewritten as
\[ 0 \longrightarrow \varpi_1(G) \longrightarrow \tG
\longrightarrow G \longrightarrow \varpi_0(G)
\longrightarrow 0, \]
where $\tG$ denotes the profinite universal cover of
$G^{\cF} := \Ker(G \to \varpi_0(G))$.

The pair $(\cC,\cF)$ satisfies the lifting property in view of
\cite[Thm.~1.1]{Br15}; thus, we may again use 
the constructions and results of Section \ref{sec:homotopy}.

\begin{lemma}\label{lem:divisible}
Let $G \in \Pro(\cC)$ be divisible. 

\begin{enumerate}

\item[{\rm (i)}] $G[n]$ is profinite for any $n \geq 1$.

\item[{\rm (ii)}] $\varpi_0(G) = 0$.

\item[{\rm (iii)}] $\tG$ is the limit of the filtered inverse 
system $(G,n_G)_{n \geq 1}$, 
where the positive integers are ordered by divisibility.
Also, $\tG$ is uniquely divisible.

\item[{\rm (iv)}] $\varpi_1(G) = \lim_{\leftarrow} G[n]$ 
(limit over the above system). Moreover, we have
$\varpi_1(G)/n \varpi_1(G) \cong G[n]$ for any $n \geq 1$.

\item[{\rm (v)}] $\varpi_i(G) = 0$ for any $i \geq 2$.

\end{enumerate}

\end{lemma}

\begin{proof}
(i) Let $G = \lim_{\leftarrow} G_i$, where the $G_i$
are algebraic groups and the projections $G \to G_i$
are epimorphisms. Then the induced map 
$G[n] \to \lim_{\leftarrow} G_i[n]$ is a monomorphism.
Moreover, each $G_i$ is divisible (as a quotient of $G$); 
thus, $G_i[n]$ is finite for dimension reasons. So 
$\lim_{\leftarrow} G_i[n]$ is profinite.

(ii) Consider an epimorphism $G \to H$, where $H \in \cF$.
Then $H$ is divisible (as a quotient of $G$) and torsion
(as a finite group scheme), hence zero. This yields the assertion.

(iii) Let $G' := \lim_{\leftarrow} G$ (limit over the above
system). For any $H \in \cC$ and $i \geq 0$, we have
\[ \Ext^i_{\Pro(\cC)}(G',H) \cong 
\lim_{\to} \Ext^i_{\Pro(\cC)}(G,H) \]
in view of \cite[V.2.3.9]{DG}. Assume that $H \in \cF$;
then we may choose an integer $n \geq 1$ such that 
$n_H = 0$. Thus, $\Ext^i_{\Pro(\cC)}(G,H)$ is killed 
by $n$, and hence $\Ext^i_{\Pro(\cC)}(G',H) = 0$. 
Using Lemma \ref{lem:sc}, it follows that the adjunction 
map $CQ(G') \to G'$ is an isomorphism. 

The projection $\pi : G' \to G$ associated with $n = 1$, 
lies in an exact sequence
\begin{equation}\label{eqn:ext} 
0 \longrightarrow \lim_{\leftarrow} G[n] \longrightarrow
G' \stackrel{\pi}{\longrightarrow} G \longrightarrow 0, 
\end{equation}
where $\lim_{\leftarrow} G[n]$ is profinite. Thus,
$\pi$ induces an isomorphism $CQ(G') \to CQ(G) = \tG$.
So we may identify $G'$ with $\tG$. Then (\ref{eqn:ext})
is identified with the universal profinite extension of $G$,
in view of Lemma \ref{lem:unique}.

(iv) The first assertion has just been proved; the second
one follows from Lemma \ref{lem:tors} in view of the
vanishing of $\varpi_0(G)$.

(v) By Lemma \ref{lem:tors} again, the profinite group scheme
$\varpi_i(G)$ is uniquely divisible for any $i \geq 2$.
As a consequence, every finite quotient of $\varpi_i(G)$
is divisible, hence zero. This yields the assertion.
\end{proof}

We may now prove a large part of our main result:

\begin{theorem}\label{thm:perfect}
Assume that $k$ is perfect.

\begin{enumerate}

\item[{\rm (i)}] We have $\varpi_i = 0$ for all $i \geq 2$;
equivalently, $\varpi_1$ is left exact.

\item[{\rm (ii)}] The cosection functor 
$C : \Pro(\cC)/\Pro(\cF) \to \Pro(\cC)$
is exact.

\item[{\rm (iii)}] The profinite universal cover 
$\tG$ has projective dimension at most $1$, 
for any $G \in \Pro(\cC)$.

\end{enumerate}

\end{theorem}

\begin{proof}
(i) In view of the homotopy exact sequence and the fact
that $\varpi_i$ commutes with filtered inverse limits, 
it suffices to show that $\varpi_i(G) = 0$ for any 
$G \in \cC$ and any $i \geq 2$. This follows from Lemma
\ref{lem:divisible} when $G$ is an abelian variety.
On the other hand, when $G \in \cL$, we have
$\varpi_i(G) =  \pi_i^{\cL,\cF}(G)$
in view of Remark \ref{rem:comp} and Lemma \ref{lem:CL}.
So the assertion follows from \cite[V.3.6.8]{DG} in that 
case. In the general case, just recall that every $G \in \cC$ 
is an extension of an abelian variety by a linear algebraic 
group.

(ii) This is just a reformulation of (i) (see Lemma 
\ref{lem:vanishing}).

(iii) By the main result of \cite{Br17a}, the category 
$\cC/\cF$ has homological dimension $1$; hence the same
holds for the category 
$\Pro(\cC)/\Pro(\cF) \cong \Pro(\cC/\cF)$
(see e.g. \cite[Prop.~2.12, Lem.~2.15]{Br17b}). As $C$ sends 
projectives to projectives, this yields the assertion.
\end{proof}

\begin{remark}\label{rem:profinite}
Returning to an arbitrary ground field $k$, consider the
full subcategory $\cE$ of $\cC$ with objects the finite
\'etale group schemes. Then $\cE$ is a Serre subcategory
of $\cF$; moreover, the pair $(\cC,\cE)$ satisfies the lifting
property if and only if $k$ is perfect (see e.g.
\cite[Thm.~1.1, Rem.~3.3]{Br15}). The functors
\[ \pi_i := \pi_i^{\cC,\cE} : 
\Pro(\cC) \longrightarrow  \Pro(\cE) \]
are the ``pro-\'etale homotopy functors'', considered in
\cite[V.3.4.1]{DG} for affine group schemes over
perfect fields; note that $\pi_0(G) = G/G^0$
for any $G \in \cC$, where $G^0$ denotes the 
neutral component (see e.g. \cite[II.5.1]{DG}).
The functor $\pi_0^{\cF,\cE} : \Pro(\cF) \to \Pro(\cE)$ 
is exact in view of \cite[V.3.1.5]{DG}; using Lemma 
\ref{lem:comp}, this yields natural isomorphisms
\[ \pi_i(G) \stackrel{\cong}{\longrightarrow}
\pi_0^{\cF,\cE}(\pi_i^{\cC,\cF}(G)) \]
for all $G \in \Pro(\cC)$ and all $i \geq 0$. 
As a consequence, the pro-\'etale fundamental group
$\pi_1$ is left exact when $k$ is perfect. 
\end{remark}

\subsection{Projective covers of abelian varieties}
\label{subsec:P(A)}

Consider an abelian variety $A$, and its projective cover 
$P(A)$ in $\Pro(\cC)$. By Proposition \ref{prop:affine}, 
we have an exact sequence in $\Pro(\cC)$
\begin{equation}\label{eqn:LP}
0 \longrightarrow L(A) \longrightarrow P(A)
\stackrel{\rho}{\longrightarrow} A \longrightarrow 0, 
\end{equation}
where $L(A)$ is affine. Also, recall that (\ref{eqn:LP}) 
is the universal affine extension of $A$, that is, 
the pushout by this extension yields an isomorphism
\begin{equation}\label{eqn:univCL} 
\Hom_{\Pro(\cL)}(L(A),G) 
\stackrel{\cong}{\longrightarrow}
\Ext^1_{\Pro(\cC)}(A,G) 
\end{equation}
for any $G \in \Pro(\cL)$. 

Next, note that an algebraic group $G$ is an object of 
${^{\perp} \Pro(\cL)}$ if and only if $G$ is anti-affine,
i.e., $\cO(G) = k$ (as follows from the affinization theorem,
see \cite[III.3.8.2]{DG}). In view of Lemma \ref{lem:ess}, 
it follows that $P(A)$ is the inverse limit of all anti-affine 
extensions of $A$. Using the affinization theorem again, 
one can deduce that the exact sequence
(\ref{eqn:LP}) is the universal affine extension of $A$
by a (not necessarily commutative) affine $k$-group 
scheme. One can also obtain a structure result for $P(A)$
by using the classification of anti-affine groups
(see \cite[Thm.~2.7]{Br09}). We will rather obtain such 
a result (Theorem \ref{thm:P(A)}) via an alternative
approach, which relates $P(A)$ to the
universal profinite cover of $A$.

Consider the exact sequence as in (\ref{eqn:exact}),
\[ 0 \longrightarrow L(A)^{\cF} \longrightarrow
L(A) \longrightarrow \varpi_0(L(A)) \longrightarrow 0. \]
Then the induced exact sequence
\[ 0 \longrightarrow \varpi_0(L(A)) \longrightarrow
P(A)/L(A)^{\cF} \longrightarrow A \longrightarrow 0 \]
is the universal profinite extension of $A$, as observed in 
Remark \ref{rem:comp}. We thus identify $\varpi_0(L(A))$
with $\varpi_1(A)$, and $P(A)/L(A)^{\cF}$ with the
profinite universal cover $\tA$. This yields an exact
sequence
\begin{equation}\label{eqn:LFP}
0 \longrightarrow L(A)^{\cF} \longrightarrow
P(A) \longrightarrow \tA \longrightarrow 0.
\end{equation}

\begin{lemma}\label{lem:acyclic}
With the above notation, 
$\varpi_i(\tA) = 0 = \varpi_i(L(A)^{\cF})$ for any 
$i \geq 0$.
\end{lemma}

\begin{proof}
Since $A$ is divisible, we have $\varpi_i(A) = 0$
for $i \geq 2$ in view of Lemma \ref{lem:divisible}.
Using the homotopy exact sequence associated
with the universal profinite extension
\[ 0 \longrightarrow \varpi_1(A) \longrightarrow
\tA \longrightarrow A \longrightarrow 0 \]
together with Lemma \ref{lem:pro},
it follows that $\varpi_i(\tA) = 0$ for $i \geq 2$
as well. Also, $\varpi_0(\tA) = 0 = \varpi_1(\tA)$
by construction.

The assertion on the $\varpi_i(L(A)^{\cF})$
follows by using the exact sequence 
(\ref{eqn:LFP}).
\end{proof}

By \cite[IV.3.1.1]{DG}, there is a unique exact
sequence in $\Pro(\cL)$
\begin{equation}\label{eqn:MLU}
 0 \longrightarrow M(A) \longrightarrow
L(A) \longrightarrow U(A) \longrightarrow 0,
\end{equation}
where $M(A)$ is of multiplicative type and $U(A)$
is unipotent; if $k$ is perfect, then (\ref{eqn:MLU})
has a unique splitting. We now investigate the 
unipotent part $U(A)$:

\begin{lemma}\label{lem:U(A)}

\begin{enumerate}

\item[{\rm (i)}] There is an isomorphism
\[ \Hom_{\Pro(\cL)}(U(A),\bG_a) \cong
H^1(A,\cO_A) \]
which is compatible with the action of
$\End_{\cC}(\bG_a)$. 

\item[{\rm (ii)}] If $\charac(k) = 0$, then $U(A)$ 
is the unipotent group with Lie algebra dual of 
$H^1(A,\cO_A)$. In particular, $\dim(U(A)) = \dim(A)$.

\item[{\rm (iii)}]  If $\charac(k) > 0$, then $U(A)$ is profinite.

\end{enumerate}

\end{lemma}

\begin{proof}
(i) Since $\Hom_{\Pro(\cL)}(M(A),\bG_a) = 0$,
the exact sequence (\ref{eqn:MLU}) yields 
an isomorphism 
\[ \Hom_{\Pro(\cL)}(U(A),\bG_a) 
\stackrel{\cong}{\longrightarrow}
\Hom_{\Pro(\cL)}(L(A),\bG_a). \]
The latter is naturally isomorphic to 
$\Ext^1_{\Pro(\cC)}(A,\bG_a)$
in view of the isomorphism (\ref{eqn:univCL}). 
Moreover, we have natural isomorphisms
\[ \Ext^1_{\Pro(\cC)}(A,\bG_a) \cong
\Ext^1_{\cC}(A,\bG_a) \cong
H^1(A,\cO_A) \]
(see e.g. \cite[III.17]{Oort}).

(ii) This follows from (i) combined with
\cite[IV.2.4.2]{DG}.

(iii) Assume that $U(A)$ is not profinite. By
\cite[V.3.2.5]{DG}, there exists an epimorphism
$U(A) \to \bG_a$. This yields a monomorphism
$\End_{\cC}(\bG_a) \to H^1(A,\cO_A)$,
a contradiction since the right-hand side is
a finite-dimensional $k$-vector space.
\end{proof}

Next, we describe the part of multiplicative
type, $M(A)$. By Cartier duality (see
\cite[IV.1.3.6]{DG}), this amounts to 
determining the character group 
$X(M(A))$ as a module under the absolute
Galois group, $\Gamma = \Gal(k_s/k)$,
where $k_s$ denotes a separable closure
of $k$.

\begin{lemma}\label{lem:T(A)}

\begin{enumerate}

\item[{\rm (i)}] If $k$ is perfect, then
$X(M(A)) \cong X(L(A)) \cong \wA(k_s)$
as Galois modules, where $\wA$ denotes 
the dual abelian variety of $A$.

\item[{\rm (ii)}] For an arbitrary field $k$, we have
\[ X(L(A)^{\cF}) \cong X(M(A)) \otimes_{\bZ} \bQ
\cong X(L(A)) \otimes_{\bZ} \bQ \cong
\wA(k_s) \otimes_{\bZ} \bQ \]
as Galois modules.
\end{enumerate}

\end{lemma}

\begin{proof}
(i) Recall that $L(A) \cong M(A) \times U(A)$.
In view of the isomorphism (\ref{eqn:univCL}),
this yields a natural isomorphism for any torus $T$
\[ \Hom_{\Pro(\cL)}(M(A),T) \cong \Ext^1_{\cC}(A,T). \]
By the Weil-Barsotti formula (see e.g.
\cite[III.17,III.18]{Oort}), there is a natural
isomorphism
\[ \Ext^1_{\cC}(A,T) \cong
\Hom^{\Gamma}(X(T), \wA(k_s)). \]
Combining these isomorphisms yields the
statement by using Cartier duality.

(ii) If $k$ is perfect, then 
$L(A)^{\cF} \cong M(A)^{\cF} \times U(A)^{\cF}$
with an obvious notation. Thus, 
$X(L(A)^{\cF}) \cong X(M(A)^{\cF})$.
In view of (i) combined with \cite[V.3.5.2]{DG},
the Galois module $X(M(A)^{\cF})$ is
the quotient of $\wA(k_s)$ by its torsion
subgroup. As $\wA(k_s)$ is divisible,
this yields the statement.

Next, if $k$ is imperfect (in particular,
of characteristic $p > 0$), then 
$L(A)^{\cF} \subset M(A)$ by
Lemma \ref{lem:U(A)} (iii); in particular,
$L(A)^{\cF}$ is of multiplicative type. Also,
$\varpi_0(L(A)^{\cF}) = 0 = \varpi_1(L(A)^{\cF})$
by Lemma \ref{lem:acyclic}. In view of
\cite[V.3.5.2]{DG} again, it follows that 
$X(L(A)^{\cF})$ is a $\bQ$-vector space. Thus,
so is $\Hom_{\Pro(\cL)}(L(A)^{\cF},T)$ for any 
torus $T$. Since $M(A)/L(A)^{\cF}$ is profinite, 
this yields a natural isomorphism
\[ \Hom_{\Pro(\cL)}(M(A),T) \otimes_{\bZ} \bQ
\cong \Hom_{\Pro(\cL)}(L(A)^{\cF},T). \]

Also, the exact sequence (\ref{eqn:MLU}) 
yields an exact sequence
\[ 0 \longrightarrow \Hom_{\Pro(\cL)}(L(A),T)  
\longrightarrow \Hom_{\Pro(\cL)}(M(A),T) 
\longrightarrow \Ext^1_{\Pro(\cL)}(U(A),T). \]
Recall that $U(A)$ is a filtered inverse limit of
unipotent algebraic groups $U_i$. Then
\[ \Ext^1_{\Pro(\cL)}(U(A),T) \cong
\lim_{\to} \Ext^1_{\Pro(\cL)}(U_i,T) \]
in view of \cite[V.2.3.9]{DG}. Moreover,
since each $U_i$ is killed by a power of $p$,
so is each $\Ext^1_{\Pro(\cL)}(U_i,T)$.
As a consequence, we obtain a natural 
isomorphism
\[ \Hom_{\Pro(\cL)}(M(A),T) \otimes_{\bZ} \bQ
\cong
\Hom_{\Pro(\cL)}(L(A),T) \otimes_{\bZ} \bQ,\]
and hence a natural isomorphism
\[ \Hom_{\Pro(\cL)}(L(A)^{\cF},T) \cong
\Hom_{\Pro(\cL)}(L(A),T) \otimes_{\bZ} \bQ. \]
Arguing as in (i) completes the proof.  
\end{proof}

We may summarize the main results of this 
subsection in the following:

\begin{theorem}\label{thm:P(A)}
Let $A$ be an abelian variety over a field $k$
with characteristic $p \geq 0$ and separable closure
$k_s$.

\begin{enumerate}

\item[{\rm (i)}] The universal profinite cover $\tA$
is the limit of the filtered inverse system of multiplication
maps $(A,n_A)_{n \geq 1}$. 

\item[{\rm (ii)}] The exact sequence {\rm (\ref{eqn:LFP})},
$0 \to L(A)^{\cF} \to P(A) \to \tA \to 0$,
is a projective resolution of $\tA$.

\item[{\rm (iii)}] If $p = 0$ then 
$L(A)^{\cF} = M(A)^{\cF} \times U(A)$,
where $M(A)^{\cF}$ is the group of multiplicative type
with character group $\wA(k_s) \otimes_{\bZ} \bQ$,
and $U(A)$ is the unipotent group with Lie algebra 
dual of $H^1(A,\cO_A)$.

\item[{\rm (iv)}] If $p > 0$ then 
$L(A)^{\cF} = M(A)^{\cF}$,
where the latter is defined as above.

\end{enumerate}

\end{theorem}

\begin{proof}
All the assertions follow from Lemmas \ref{lem:divisible}, 
\ref{lem:U(A)} and \ref{lem:T(A)}, except for 
the projectivity of $L(A)^{\cF}$ in $\Pro(\cC)$, or
equivalently in $\Pro(\cL)$. If $p > 0$, then the group
$L(A)^{\cF}$ is of multiplicative type and its character
group is a $\bQ$-vector space, hence the desired assertion 
by \cite[V.3.5.2]{DG}. If $p = 0$, then we use in addition
the fact that every unipotent group is projective
in $\Pro(\cL)$ (see e.g. \cite[V.3.6.5]{DG}.
\end{proof}

\subsection{Structure of indecomposable projectives}
\label{subsec:indproj}

We still consider an arbitrary ground field $k$, of
characteristic $p \geq 0$.

\begin{proposition}\label{prop:projind}
The indecomposable projectives of $\Pro(\cC)$ 
are exactly:

\begin{enumerate}

\item[{\rm (i)}] the $P(A)$, where $A$ is a simple
abelian variety,

\item[{\rm (ii)}] the universal profinite covers 
of the simple tori,

\item[{\rm (iii)}] the additive group $\bG_a$ if $p = 0$, 
resp.~the universal profinite cover of the Witt group scheme
$W : = \lim_{\leftarrow} W_n$ if $p > 0$,

\item[{\rm (iv)}] the indecomposable projectives of
$\Pro(\cF)$.

\end{enumerate}

\end{proposition}

\begin{proof}
Applying Corollary \ref{cor:indproj} to the pair $(\cC,\cF)$, we see
that the indecomposable projectives of $\Pro(\cC)$ are exactly 
those of $\Pro(\cF)$ and the universal profinite covers $\tP$, 
where $P$ is an indecomposable projective of $\Pro(\cC/\cF)$. 
Also, every object of $\cC/\cF$ has finite length (see 
\cite[Prop.~3.2]{Br17a}). In view of \cite[V.2.4.6]{DG}, it follows 
that every indecomposable projective of $\Pro(\cC/\cF)$
is the projective cover of a simple object of $\cC/\cF$, unique 
up to isomorphism.

Next, the simple objects of $\cC/\cF$ are exactly $\bG_a$,
the simple tori and the simple abelian varieties (see 
\cite[Prop.~3.2]{Br17a} again). Moreover,
every torus is projective in $\cC/\cF$, and hence in 
$\Pro(\cC/\cF)$; also, $\bG_a$ is projective if and only if
$p = 0$ (see \cite[Thm.~5.14]{Br17a}. The universal profinite
cover of a torus $T$ is the group of multiplicative type 
with character group $X(T) \otimes_{\bZ} \bQ$, in view of
\cite[V.3.5.2]{DG}. Also, $\widetilde{\bG_a} = \bG_a$ if 
$p = 0$, as follows e.g. from Lemma \ref{lem:divisible}. 
If $p > 0$ and $k$ is perfect, then the projective cover 
of $\bG_a$ in $\cL$ (or equivalently, in $\cC$) is 
the universal profinite cover $\tW$ (see \cite[V.3.7.5]{DG}); 
equivalently, $W$ is the projective cover of $\bG_a$ in 
$\Pro(\cC/\cF)$. But the category $\cC/\cF$ is invariant under 
base change by purely inseparable field extensions (see 
\cite[Thm.~3.11]{Br17a}); moreover, $W$ is obtained by base 
change of a group scheme of finite type over $\bZ$, and hence 
makes sense over an arbitrary field $k$. Thus, $\tW$ is the 
projective cover of $\bG_a$ in that setting, too.
\end{proof}

\begin{remark}\label{rem:indproj} 
We now describe the indecomposable projectives of 
the profinite category $\Pro(\cF)$ in terms of those of
the pro-\'etale category $\Pro(\cE)$. For this, we may 
assume that $p > 0$, since $\cF = \cE$ if $p = 0$.

We will adapt the arguments in the proof of Proposition 
\ref{prop:projind}
twice. First, consider the pair $(\cF,\cI)$, where $\cI$ 
denotes the full subcategory of $\cF$ consisting of the 
infinitesimal algebraic groups; then $\cI$ is a Serre 
subcategory of $\cF$, and the pair $(\cF,\cI)$ satisfies 
the lifting property in view of \cite[Lem.~2.2]{Br17a}. 
Also, the quotient category $\cF/\cI$ is equivalent to 
the category $\cE$ of \'etale algebraic groups, 
by assigning to any finite algebraic group its largest 
\'etale quotient. It follows that the functor
\[ \pi_0^{\cF,\cE} : \Pro(\cF) \longrightarrow \Pro(\cE) \]
yields an equivalence of categories 
$\Pro(\cF/\cI) \cong \Pro(\cE)$. Thus,
\emph{the indecomposable projectives of $\Pro(\cF)$
are exactly those of $\Pro(\cI)$ and the 
universal pro-infinitesimal covers $\tP$, where
$P$ is an indecomposable projective of $\cE$.}

Next, consider the pair $(\cI,\cI_m)$, where 
$\cI_m$ denotes the full subcategory of $\cI$
consisting of (infinitesimal algebraic) groups of 
multiplicative type. Then again, $\cI_m$ is
a Serre subcategory; moreover, 
$\cI/\cI_m \cong \cI_u$, the full subcategory
of $\cI$ consisting of unipotent groups
(see \cite[IV.3.1.1]{DG}). Also, $\cI_u$ has 
a unique simple object $\alpha_p$, the kernel 
of the Frobenius endomorphism of $\bG_a$
(see \cite[IV.2.2.5]{DG}).

We now show that the pair $(\cI,\cI_m)$ satisfies 
the lifting property. Consider an epimorphism 
$f : G \to H$ in $\cI$, where $H$ is multiplicative. 
Denote by $M$ the largest multiplicative subgroup 
of $G$; then $G/M$ is unipotent, hence so is $H/f(M)$. 
It follows that $H/f(M) = 0$, i.e., the composition
$M \to G \to H$ is an epimorphism as well.

As a consequence, we see that 
\emph{the indecomposable projectives of $\Pro(\cI)$
are exactly those of $\Pro(\cI_m)$ and the universal 
multiplicative cover $\tP$, where $P$ is the
projective cover of $\alpha_p$ in $\cI_u$.}

The above results take a much simpler form when 
$k$ is perfect: then we have an equivalence of
categories 
\[ \cF \cong \cI_m \times \cI_u \times \cE \]
in view of \cite[IV.3.5.9]{DG}. Thus, the indecomposable 
projectives of $\Pro(\cF)$ are exactly those of 
$\Pro(\cI_m)$, $\Pro(\cI_u)$ and $\Pro(\cE)$. 
\end{remark}

\subsection{Field extensions}
\label{subsec:field}

For any field extension $k'/k$, we denote by 
\[ \otimes_k \, k' : \cC_k \longrightarrow \cC_{k'}, 
\quad G \longmapsto G_{k'} \] 
the associated base change functor. Then 
$\otimes_k k'$ is exact and faithful; hence it
extends uniquely to an exact functor
$\Pro(\cC_k) \to \Pro(\cC_{k'})$
which commutes with filtered inverse limits
(see e.g. \cite[Prop.~6.1.9, Cor.~8.6.8]{KS}).
We still denote this extension by $\otimes_k \, k'$.

\begin{lemma}\label{lem:change}
The functor $\otimes_k \, k' : \Pro(\cC_k) \to \Pro(\cC_{k'})$ 
is faithful. If $k'/k$ is separable algebraic, then 
$\otimes_k \, k'$ sends projectives to projectives.
\end{lemma}

\begin{proof}
Let $X, Y \in \Pro(\cC_k)$ and 
$f \in \Hom_{\Pro(\cC_k)}(X,Y)$ such that
$f_{k'} = 0$. Then $\Im(f_{k'}) = 0$.
Since $\otimes_k \, k'$ is exact, this means that
$\Im(f)_{k'} = 0$. Let $Z := \Im(f)$, then
$Z = \lim_{\leftarrow} Z_i$ (filtered inverse limit),
where $Z_i \in \cC_k$ and $Z \to Z_i$ is an
epimorphism for all $i$. Thus, $Z_{k'}$
is the filtered inverse limit of the $(Z_i)_{k'}$,
and $Z_{k'} \to (Z_i)_{k'}$ is an epimorphism
for all $i$ as well. As $Z_{k'} = 0$, it follows that 
$(Z_i)_{k'} = 0$ for all $i$. So $Z_i = 0$
and $Z = 0$, that is, $f = 0$.
This proves that $\otimes_k \, k'$ is faithful. 

Next, assume that $k'/k$ is separable algebraic
and let $P \in \Pro(\cC_k)$ be projective.
To show that $P_{k'}$ is projective in $\Pro(\cC_{k'})$, 
it suffices to check that given an epimorphism 
$f: G \to H$  and a morphism $g : P_{k'} \to H$,
where $G,H \in \cC_{k'}$, 
there exists a morphism $h : P_{k'} \to G$ in $\cC_{k'}$ 
such that $g = f \circ h$ (see \cite[V.2.3.5]{DG}). As above, 
we have $P = \lim_{\leftarrow} P_i$ (filtered inverse limit),
where $P_i \in \cC_k$ and $P \to P_i$ is an
epimorphism for all $i$. So $g$ lies in
\[ \Hom_{\Pro(\cC_{k'})}(P_{k'}, H)
= \Hom_{\Pro(\cC_{k'})}(\lim_{\leftarrow}(P_i)_{k'}, H)
= \lim_{\to} \Hom_{\Pro(\cC_{k'})}((P_i)_{k'}, H). \]
Thus, $g$ is represented by a morphism
$g_i : (P_i)_{k'} \to H$ for some $i$. Since the schemes
$G,H,(P_i)_{k'}$ are of finite type over $k'$, 
the morphisms $f: G \to H$ and $g_i : (P_i)_{k'} \to H$
are ``defined over some finite subextension $K/k$'',
i.e., there exist such a subextension and
morphisms $f_K : G_K \to H_K$, 
$(g_i)_K : (P_i)_K \to H_K$ in $\cC_K$ such that
$f = f_K \otimes_K \, k'$ and 
$g_i = (g_i)_K \otimes_K \, k'$. Then
\[ (g_i)_K \in \Hom_{\cC_K}((P_i)_K, H_K)
= \Hom_{\cC_k}(P_i, \R_{K/k}(H_K)), \]
where $\R_{K/k}$ denotes the Weil restriction
(see e.g. \cite[I.1.6.6]{DG} or \cite[App.~B]{CGP}). 
As $K/k$ is finite and separable and 
$f_K : G_K \to H_K$ is an epimorphism, the map 
$\R_{K/k}(f_K) : \R_{K/k}(G_K) \to \R_{K/k}(H_K)$
is an epimorphism as well (see \cite[III.5.7.9]{DG}). 
Since $P$ is projective, it follows that
$(g_i)_K$ lifts to a morphism  
\[ (f_j)_K \in \Hom_{\cC_k}(P_j, \R_{K/k}(G_K))
= \Hom_{\cC_K}((P_j)_K, G_K) \]
for some $j$. This yields a lift 
$f_j \in  \Hom_{\cC_{k'}}((P_j)_{k'}, G_{k'})$
of $g_i$, and in turn the desired lift 
$f \in  \Hom_{\cC_k'}(P_{k'}, G_{k'})$ of $g$.
\end{proof}

\begin{remark}\label{rem:change}
In the setting of affine group schemes, the fact that 
the base change functor $\otimes_k k'$ preserves projectives 
for any separable algebraic extension $k'$ of $k$ is due to 
Demazure and Gabriel (see \cite[V.3.2.1]{DG}).
For arbitrary group schemes, this fact is stated and
used in \cite[p.~437]{Milne}, but the argument sketched 
there is flawed.
\end{remark}

We may now complete the proof of the main theorem:

\begin{proposition}\label{prop:separable}
For any $i \geq 0$, the functors $\pi_i^{\cC,\cL}$ and
$\varpi_i$ commute with base change under separable 
algebraic field extensions. Moreover, the same holds 
for the universal affine and profinite covers.
\end{proposition}

\begin{proof}
The restriction of $\pi_0^{\cC,\cL}$ to $\cC$ is the affinization 
functor $\cC \to \cL$, which commutes with base change under
arbitrary field extensions (see e.g. \cite[III.3.8.1]{DG}).
Thus, so does $\pi_0^{\cC,\cL}$, since it commutes with filtered 
inverse limits. By Lemma \ref{lem:change}, it follows that 
$\pi_i^{\cC,\cL}$ commutes with base change under separable 
algebraic field extensions for any $i \geq 1$. In view of Lemma
\ref{lem:unique}, the same holds for the universal affine cover.

We now show that $\varpi_0$ (the largest
profinite quotient) commutes with $\otimes_k k'$, 
where $k'/k$ is any separable algebraic field extension; 
this will imply the assertions on the profinite homotopy
groups and profinite universal cover by arguing as above. 
For any $X \in {^{\perp}\Pro(\cF_k)}$, we have to check 
that $X_{k'} \in {^{\perp}\Pro(\cF_{k'})}$, i.e., 
$\Hom_{\cC_{k'}}(X_{k'},Y) = 0$ for any 
$Y \in \cF_{k'}$. But this follows by a Weil restriction
argument as in the proof of Lemma \ref{lem:change}.

More specifically, let $X = \lim_{\leftarrow} X_i$
(filtered inverse limit), where $X_i \in \cC_k$
and the natural map $X \to X_i$ is an epimorphism
for all $i$. Then $X_{k'} = \lim_{\leftarrow} (X_i)_{k'}$
(filtered inverse limit), where $(X_i)_{k'} \in \cC_{k'}$
and the natural map $X_{k'} \to (X_i)_{k'}$ is 
an epimorphism for all $i$ as well. Thus, for any
morphism $f : X_{k'} \to  Y$, where $Y \in \cF_{k'}$,
there exists $i$ such that $f$ is the composition
$X_{k'} \to (X_i)_{k'} \to Y$ for some morphism
$f_i : (X_i)_{k'} \to Y$. In turn, there exist a finite
subextension $K/k$ and a morphism 
$(f_i)_K : (X_i)_K \to Y_K$ in $\cC_K$, 
such that $f_i = (f_i)_K \otimes_K \, k'$. 
We now have
\[ (f_i)_K \in \Hom_{\cC_K}((X_i)_K, Y_K)
= \Hom_{\cC_k}(X_i, \R_{K/k}(Y_K)). \]
Moreover, $\R_{K/k}(Y_K) \in \cF_k$, since
$Y$ is a finite $k'$-group scheme and hence 
$Y_K$ is a finite $K$-group scheme.
It follows that
$\Hom_{\cC_k}(X_i, \R_{K/k}(Y_K)) = 0$,
as $X \in {^{\perp}\Pro(\cF_k)}$. Thus, $(f_i)_K = 0$, 
so that $f_i = 0$ and $f = 0$.
\end{proof}

\begin{remark}\label{rem:separable}
One may check similarly that the functors 
$\pi_i^{\cF,\cI}$ and the universal pro-infinitesimal
cover (considered in Remark \ref{rem:indproj}) also
commute with base change under separable
algebraic field extensions. Indeed, being infinitesimal
is preserved under Weil restriction associated with
finite separable field extensions.

Likewise, the functors $\pi_i^{\cI,\cI_m}$ and the 
universal multiplicative cover commute with such 
base change, since being multiplicative is preserved 
under Weil restriction as above.
\end{remark}

By Proposition \ref{prop:separable}, the profinite 
fundamental group $\varpi_1$ commutes with base 
change under algebraic field extensions in 
characteristic $0$. Yet this does not extend to 
an imperfect ground field, see Example 
\ref{ex:inseparable} (iii) below. To remedy this, 
we now recall the definition of the prime-to-$p$ part 
of $\varpi_1$, and show that it satisfies the assertions 
of the main theorem.

Every finite group scheme $G$ decomposes into a
product $G_p \times G_{p'}$, where $G_p$ is a
$p$-group, and $G_{p'}$ has order prime to $p$; 
moreover, $G_{p'}$ is \'etale. This decomposition
is clearly functorial, and yields an equivalence of categories
$\cF \cong \cF_p \times \cF_{p'}$ with an obvious
notation. In turn, we obtain an equivalence of
categories 
\[ \Pro(\cF) \cong \Pro(\cF_p) \times \Pro(\cF_{p'}), \]
where every object of $\Pro(\cF_{p'})$ is pro-\'etale.
Composing the resulting exact functor 
$\Pro(\cF) \to \Pro(\cF_{p'})$ (the prime-to-$p$ part) with 
the profinite homotopy functors $\varpi_i$, we obtain functors
\[ \varpi_i^{(p')} : \Pro(\cC) \longrightarrow \Pro(\cF_{p'}). \]

\begin{proposition}\label{prop:p'}
With the above notation and assumptions, the functor 
$\varpi_1^{(p')}$ is left exact and commutes with base change 
under algebraic field extensions.
If $k$ is algebraically closed and $G$ is a smooth connected
algebraic group, then $\varpi_1^{(p')}(G)$ is the prime-to-$p$ 
part of the \'etale fundamental group of the scheme $G$.
\end{proposition}

\begin{proof}
To show the first assertion, it suffices to check that 
$\varpi_1^{(p')}$ commutes with purely inseparable
field extensions, in view of Theorem \ref{thm:perfect} 
and Proposition \ref{prop:separable}. We may identify
the prime-to-$p$ functor $\Pro(\cF) \to \Pro(\cF_{p'})$ 
with the quotient functor 
$Q^{\cF,\cF_p}: \Pro(\cF) \to \Pro(\cF)/\Pro(\cF_p)$; 
moreover, the pair $(\cC,\cF_p)$ satisfies the lifting 
property (see \cite[Lem.~3.1]{Br17b}). Thus, $\varpi_i^{(p')}(G)$ 
is identified with the image of $\varpi_i(G)$ in $\cC/\cF_p$
for any $G \in \Pro(\cC)$ (Lemma \ref{lem:compatible}). 
Moreover, the category $\cC/\cF_p$ is invariant
under base change by purely inseparable extensions,
in view of \cite[Thm.~3.17]{Br17b}; thus so is its
torsion subcategory, $\cF/\cF_p$. This implies the
desired statement.

The second assertion follows from the fact that every 
\'etale Galois cover of the scheme $G$ has the structure 
of a smooth commutative algebraic group, unique up to 
the choice of the neutral element (see e.g. 
\cite[Prop.~1.1]{BS}). 
\end{proof}

\begin{remark}\label{rem:proetale}
To obtain a version of the profinite fundamental group which 
commutes with all algebraic field extensions, one may
also consider the quotient category of $\cC$ by the Serre 
subcategory $\cI$ of infinitesimal algebraic groups. 
We may view $\cC/\cI$ as the category of algebraic
groups up to purely inseparable isogeny,
or alternatively as that of quasi-algebraic groups in
the sense of \cite{Se60} (see also \cite[V.3.4.5]{DG}).
The pair $(\cC,\cI)$ satisfies the lifting property (see e.g.
\cite[Lem.~2.2]{Br17a}); in view of Lemma \ref{lem:compatible},
it follows that $\pi_i^{\cC/\cI,\cF/\cI}(G)$ is the image
of $\varpi_i(G)$ in $\Pro(\cF/\cI)$, for any $G \in \Pro(\cC)$
and any $i \geq 0$. But $\Pro(\cF/\cI) \cong \Pro(\cE)$
via the largest pro-\'etale quotient functor $\pi_0^{\cC,\cE}$;
also, $\cC/\cI$ and its subcategory $\cF/\cI$ are invariant 
under base change by purely inseparable field extensions 
(see \cite[Thm.~3.17]{Br17b}). As a consequence, we obtain
functors 
\[ \pi_i^{\cC/\cI,\cE} : \Pro(\cC/\cI) \longrightarrow \Pro(\cE) \]
which commute indeed with algebraic field extensions.

If $k$ is perfect, then $\cF$ is naturally equivalent
to $\cI \times \cE$; as a consequence, 
$\Pro(\cF) \cong \Pro(\cI) \times \Pro(\cE)$ 
and this identifies the quotient functor 
\[ Q^{\cF,\cI} : \Pro(\cF) \longrightarrow \Pro(\cF/\cI) \] 
with the corresponding projection $\Pro(\cF) \to \Pro(\cE)$. 
It follows that the composite functor
$\pi_i^{\cC/\cI,\cE} \circ Q^{\cC,\cI}$
is identified with the pro-\'etale homotopy functor
$\pi_i$ discussed in Remark \ref{rem:profinite}.
Thus, $\pi_1^{\cC/\cI,\cE}$ is left exact and its 
prime-to-$p$ part is $\varpi_1^{(p')}$.

The latter assertion extends to an imperfect
field $k$, since $\pi_1^{\cC/\cI,\cE}$ may be
identified with the pro-\'etale fundamental
group over its perfect closure.
\end{remark}

\begin{examples}\label{ex:inseparable}
(i) The functor $\varpi_0$ does not commute with 
base change under purely inseparable field extensions.
Consider indeed an imperfect field $k$, and choose 
$t \in k \setminus k^p$. Let $G$ denote the
kernel of the morphism
\[ \bG_a \times \bG_a \longrightarrow \bG_a, 
\quad (x,y) \longmapsto x^p - t y^p. \]
Then $G$ is connected and reduced; thus,
$\varpi_0(G)$ is connected and reduced as well,
hence zero. Let $k' := k(t^{1/p})$, then the map
$(x,y) \mapsto (x, x - t^{1/p} y)$
yields an isomorphism
$G_{k'} \cong \bG_{a,k'} \times \alpha_{p,k'}$, 
where $\alpha_{p,k'}$ denotes the kernel of the 
Frobenius endomorphism
\[ F : \bG_{a,k'} \longrightarrow \bG_{a,k'}, \quad
x \longmapsto x^p. \] 
Thus, $\varpi_0(G_{k'}) \cong \alpha_{p,k'}$.

\smallskip

\noindent
(ii) The functor $\varpi_1$ does not commute with 
base change under purely inseparable field extensions
either. Consider indeed a smooth connected algebraic 
group $G$ and a finite group scheme $H$. Then 
$\varpi_0(G) = 0$, hence we obtain canonical isomorphisms
\[ \Hom_{\Pro(\cF)}(\varpi_1(G),H)
\cong \Ext^1_{\Pro(\cC)}(G,H) 
\cong \Ext^1_{\cC}(G,H). \]
If $\varpi_1$ commutes with base change under 
an extension of fields $k'/k$, then the natural map
\[ \Ext^1_{\cC}(G,H) \longrightarrow 
\Ext^1_{\cC_{k'}}(G_{k'},H_{k'}) \]
is injective in view of the above isomorphisms
and the faithfulness of $\otimes_k \, k'$
(Lemma \ref{lem:change}).

Now assume that $k$ is separably closed,
but not algebraically closed; then there exist
nontrivial $k$-forms of $\bG_a$, and
$\Ext^1_{\cC}(G,\bG_m) \neq 0$ for any such form $G$ 
(see \cite[Lem.~9.4]{Totaro}). As $G$ is killed by $p$, so 
is $\Ext^1_{\cC}(G,\bG_m)$. It follows that the natural map
\[ \Ext^1_{\cC}(G,\mu_p) \longrightarrow 
\Ext^1_{\cC}(G,\bG_m) \]
is surjective, where  $\mu_p$ denotes the kernel
of the $p$th power map of $\bG_m$.
Thus, $\Ext^1_{\cC}(G,\mu_p) \neq 0$.
On the other hand,
$\Ext^1_{\cC_{\bar{k}}}(G_{\bar{k}}, \mu_{p,\bar{k}}) 
= \Ext^1_{\cC_{\bar{k}}}(\bG_{a,\bar{k}}, \mu_{p,\bar{k}})$
vanishes in view of the structure of commutative linear 
algebraic groups over algebraically closed fields (see e.g. 
\cite[IV.3.1.1]{DG}). So $\varpi_1$ does not commute 
with the (purely inseparable) extension $\bar{k}/k$.

\smallskip

\noindent
(iii) The above examples show that the 
``pro-infinitesimal part'' of $\varpi_i$ (the largest 
pro-infinitesimal subobject) does not commute with 
base change under purely inseparable field extensions 
for $i = 0, 1$. One may wonder whether the ``pro-\'etale
part'' (the largest pro-\'etale quotient of $\varpi_i$) 
is better behaved. The answer is affirmative for $\varpi_0$, 
which commutes with arbitrary field extensions 
(see \cite[II.5.1]{DG}). Also, the answer is affirmative 
for the prime-to-$p$ part of $\varpi_1$ by Proposition 
\ref{prop:p'}. But the answer is negative for its pro-\'etale
$p$-primary part, as we now show in the case of the
additive group $\bG_a$. 

Since $\bG_a$ is killed by $p$, so are $\varpi_1(\bG_a)$ 
and its largest pro-\'etale quotient $Q$. Denoting by $\nu_p$ 
the constant $k$-group scheme associated with $\bZ/p \bZ$,
it follows that the natural map 
$\Hom_{\Pro(\cC)}(Q,\nu_p) \to
\Hom_{\Pro(\cC)}(\varpi_1(\bG_a), \nu_p)$
is an isomorphism. So it suffices to show that the formation 
of $\Hom_{\Pro(\cC)}(\varpi_1(\bG_a), \nu_p)$ does not
commute with purely inseparable field extensions. 

As in (ii) above, we have an isomorphism
\[  \Hom_{\Pro(\cC)}(\varpi_1(\bG_a), \nu_p)
\cong \Ext^1_{\cC}(\bG_a,\nu_p) \]
of modules over $\End_{\cC}(\bG_a)$. Also, recall that
$\End_{\cC}(\bG_a)$ consists of the additive polynomials,
\[ x \longmapsto a_0 \, x + a_1 \, x^p + \cdots + a_n \, x^{p^n}, \]
where $a_0, \ldots, a_n \in k$ (see e.g. \cite[II.3.4.4]{DG}).
By the proof of \cite[Prop.~1.20]{Saito}, we have  an 
``Artin-Schreier'' exact sequence
\begin{equation}\label{eqn:AS}
\End_{\cC}(\bG_a) \stackrel{\cP}{\longrightarrow} 
\End_{\cC}(\bG_a)
\longrightarrow \Ext^1_{\cC}(\bG_a,\nu_p) 
\longrightarrow 0
\end{equation}
of $\End_{\cC}(\bG_a)$-modules, where 
$\End_{\cC}(\bG_a)$ acts on its two copies by right 
multiplication, and 
$\cP(f)(x) := f(x)^p - f(x)$ for any $f \in \End_{\cC}(\bG_a)$
and $x \in \bG_a$.
We claim that the exact sequence (\ref{eqn:AS}) can also be 
obtained as follows: consider a nontrivial extension 
\[ 0 \longrightarrow \nu_p \longrightarrow G
\stackrel{x}{\longrightarrow} \bG_a \longrightarrow 0. \]
Then $G$ is smooth and unipotent; also, the composition
$G^0 \to G \to \bG_a$ is an epimorphism, where $G^0$
denotes the neutral component. It follows that $G$ is
connected, and hence is a $k$-form of $\bG_a$. By
\cite[Lem.~1.3]{Russell}, there is an exact sequence
\[ 0 \longrightarrow I \longrightarrow G 
\stackrel{y}{\longrightarrow} \bG_a \longrightarrow 0, \]
where $I$ is infinitesimal; moreover, we have $y = F^n_G$ 
for $n \gg 0$. Then the morphism
$(x,y) : G \to \bG_a \times \bG_a$ has a trivial kernel;
its cokernel is a quotient of $\bG_a \times \{ 0 \}$
for dimension reasons, and hence is isomorphic to 
$\bG_a$ in view of \cite[IV.2.1.1]{DG}. This yields 
an exact sequence
\[ 0 \longrightarrow G 
\stackrel{(x,y)}{\longrightarrow} \bG_a \times \bG_a
\stackrel{f + g}{\longrightarrow} \bG_a 
\longrightarrow 0, \]
where $f, g \in \End_{\cC}(\bG_a)$. So we may view $G$
as the zero scheme $\cV(f(x) + g(y))$ in 
$\bG_a \times \bG_a$; this identifies 
$\nu_p = \Ker(x: G \to \bG_a)$ with $\Ker(g)$. 
We may thus assume that
$g(y) = y^p - y$, so that $G = \cV(y^p - y + f(x))$.
This defines a map
\[ u : \End_{\cC}(\bG_a) \longrightarrow 
\Ext^1_{\cC}(\bG_a, \nu_p), \quad 
f \longmapsto \cV(y^p - y + f(x)), \]
which is surjective as $f = 0$ gives the trivial
extension. One may readily check that $u$ is
a morphism of $\End_{\cC}(\bG_a)$-modules;
also, $u(f) = 0$ if and only if $f(x) = h(x)^p - h(x)$
for some $h \in \End_{\cC}(\bG_a)$, that is,
$f =\cP(h)$. This completes the proof of the claim.

Clearly, we have 
$\Ker(\cP) = \Hom_{\cC}(\bG_a,\nu_p) = 0$. 
To describe $\Coker(\cP)$, we first consider the case
where $k$ is perfect. Then
\[ a \, x^{p^n} = \cP(a^{1/p} \, x^{p^{n-1}}) 
+ a^{1/p} \, x^{p^{n-1}} \]
for all $a \in k$ and all integers $n \geq 1$. It follows that
$\Coker(\cP) \cong k$ via the map
$k \to \End_{\cC}(\bG_a)$ given by scalar
multiplication. For an arbitrary field $k$, we obtain
by using a $p$-basis
\[ \Coker(\cP) \cong k \oplus 
\bigoplus_{n = 1}^{\infty} k \, x^{p^n}/k^p \, x^{p^n}. \]
In particular, the natural map $k \to \Coker(\cP)$
is not surjective if $k$ is imperfect. This shows that
$\Ext^1_{\cC}(\bG_a,\nu_p)$ does not commute
with purely inseparable field extensions.

The above construction may be interpreted in terms of
the exact sequence
\[ 0 \longrightarrow \nu_p 
\stackrel{\iota}{\longrightarrow} \bG_a
\stackrel{F - \id}{\longrightarrow} \bG_a 
\longrightarrow 0, \]
which yields an exact sequence 
\[ 0 \longrightarrow \End_{\cC}(\bG_a)/(F - \id)
\longrightarrow \Ext^1_{\cC}(\bG_a,\nu_p)
\stackrel{\iota^*}{\longrightarrow} 
\Ext^1_{\cC}(\bG_a, \bG_a), \]
where the image of $\iota^*$ is the kernel of 
$F - \id$. Since $ \End_{\cC}(\bG_a)$ is the 
noncommutative polynomial ring $k[F]$, we have 
$\End_{\cC}(\bG_a)/(F - \id) \cong k$.

If $k$ is perfect, then $\Ext^1_{\cC}(\bG_a, \bG_a)$ 
is a free module over $\End_{\cC}(\bG_a)$ acting 
on the left (see \cite[V.1.5.2]{DG}). Thus, we obtain 
an isomorphism of $\End_{\cC}(\bG_a)$-modules 
\[ \Ext^1_{\cC}(\bG_a,\nu_p) 
\cong k[F]/(F - \id) \cong k. \]
This isomorphism does not extend to an imperfect 
field $k$, as the image of $\iota^*$ may be
identified with $\bigoplus_{n = 1}^{\infty} k/k^p$.
\end{examples}

\subsection{The Milne spectral sequence}
\label{subsec:milne}

We first record a variant of a result obtained by Demazure 
and Gabriel in the setting of affine group schemes 
(see \cite[V.3.2.3]{DG}):

\begin{lemma}\label{lem:finite}
Let $k'/k$ be a separable field extension. Then there are
canonical isomorphisms for any $G \in \Pro(\cC)$, 
$H \in \cC$ and $j \geq 0$:
\[ \Ext^j_{\Pro(\cC_{k'})}(G_{k'},H_{k'}) \cong 
\lim_{\to, K} \Ext^j_{\Pro(\cC_K)}(G_K,H_K), \]
where  $K/k$ runs over the filtered direct system of finite
subextensions of $k'/k$.
\end{lemma}

\begin{proof}
We follow the argument of \cite[V.3.2.3]{DG} closely.
If $G \in \cC$ then the natural map
\[  \lim_{\to,K} \Hom_{\cC_K}(G_K,H_K) 
\longrightarrow \Hom_{\cC_{k'}}(G_{k'},H_{k'})\]
is an isomorphism by the ``principle of the finite extension''
(see e.g. \cite[I.3.2.2]{DG}).

For an arbitrary $G \in \Pro(\cC)$, consider the family
$(G_i)$ of its algebraic group quotients. Then
$G_{k'} \cong \lim_{\leftarrow,i} (G_i)_{k'}$, hence
\[ \Hom_{\Pro(\cC_{k'})}(G_{k'},H_{k'}) 
\cong 
\lim_{\to,i} \Hom_{\cC_{k'}}((G_i)_{k'},H_{k'}) 
\cong 
\lim_{\to,i} \lim_{\to,K} \Hom_{\cC_K}((G_i)_K,H_K)\]
 \[ \cong 
\lim_{\to,K} \lim_{\to,i} \Hom_{\cC_K}((G_i)_K,H_K) 
\cong
\lim_{\to,K} \Hom_{\Pro(\cC_K)}(G_K,H_K). \]
This yields the assertion for $j = 0$. Next, choose a 
projective resolution $P_{\bullet}$ of $G$ in $\Pro(\cC)$;
then $(P_{k'})_{\bullet}$ is a projective resolution of
$G_{k'}$ by Lemma \ref{lem:change}. Since we have 
\[ \Hom_{\Pro(\cC_{k'})}((P_{k'})_{\bullet}, H_{k'}) \cong 
\lim_{\to,K} \Hom_{\Pro(\cC_K)}((P_K)_{\bullet}, H_K), \]
this yields the statement by taking cohomology.
\end{proof}

Next, consider a Galois field extension $k'/k$. Then 
the profinite group 
\[ \Gamma := \Gal(k'/k) \]
acts on the group
$\Ext^j_{\Pro(\cC_{k'})}(G_{k'}, H_{k'})$ for any 
$G,H \in \Pro(\cC)$, and $j \geq 0$. If $H \in \cC$, then
this $\Gamma$-module is discrete as a consequence of Lemma 
\ref{lem:finite}. We may now state the following result, 
due to Milne when $k$ is perfect with algebraic closure  
$k'$ (see \cite[Prop., p.~437]{Milne}):

\begin{theorem}\label{thm:milne}
There is a spectral sequence
\[ H^i(\Gamma, \Ext^j_{\Pro(\cC_{k'})}(G_{k'}, H_{k'}))
\Rightarrow \Ext^{i + j}_{\Pro(\cC)}(G, H) \]
for any $G \in \Pro(\cC)$ and $H \in \cC$.
\end{theorem}

The proof will combine the approach sketched in \cite{Milne} 
with the inductive description of indecomposable projectives 
obtained in Subsection \ref{subsec:P(A)}. To simplify the notation,
we set 
\[ G_{k'} =: G', \quad H_{k'} =: H', \quad \cC_{k'} =: \cC', \quad \ldots \]

By Lemma \ref{lem:change}, the base change functor 
$\Pro(\cC) \to \Pro(\cC')$
is exact and sends projectives to projectives. Also, note that
\[ H^0(\Gamma, \Hom_{\Pro(\cC')}(G', H'))
= \Hom_{\Pro(\cC)}(G,H), \]
since this holds by Galois descent when $G \in \cC$, and 
taking $\Gamma$-invariants commutes with direct limits.
So Theorem \ref{thm:milne} will follow from the 
spectral sequence of composite functors (see 
\cite[Thm.~2.4.1]{Grothendieck}), once we show:

\begin{proposition}\label{prop:acyclic}
Let $G$ be a projective object of $\Pro(\cC)$,
and $H \in \cC$. Then the $\Gamma$-module
 $\Hom_{\Pro(\cC')}(G',H')$ is acyclic.
\end{proposition}

We start the proof of the above proposition with some 
observations and reductions. Since being acyclic is 
preserved under taking direct limits, we may assume 
that $k'/k$ is \emph{finite} by combining Lemmas 
\ref{lem:change} and \ref{lem:finite}. Also, recall that
$G \cong \prod_{i \in I} P_i$, where the $P_i$ are
indecomposable and projective. Thus,
$G' \cong \prod_{i \in I} P'_i$ and
\[ \Hom_{\Pro(\cC')}(G', H') 
\cong \bigoplus_{i \in I} 
\Hom_{\Pro(\cC')}(P'_i, H'). \]
To show the acyclicity of this $\Gamma$-module, we may
therefore assume that $G$ is \emph{indecomposable}. 
Thus, $G$ is of one of the types listed in Proposition 
\ref{prop:projind}.

Assume first that $G = P(A)$, where $A$ is a simple 
abelian variety. Then $G' = P(A')$ (the universal 
affine cover of $A'$) in view of Proposition 
\ref{prop:separable}. So the adjunction isomorphism 
(\ref{eqn:adj}) yields an isomorphism of $\Gamma$-modules
\[ \Hom_{\Pro(\cC')}(G',H') \cong \Hom_{\ucA'}(A', Q(H')), \]
where $Q := Q^{\cC', \cL'}$.
The right-hand side is a $\bQ$-vector space, and hence 
an acyclic $\Gamma$-module in view of \cite[V.3.5.1]{DG}.

Next, assume that $G$ is the universal profinite
cover of a simple torus $T$. Then $G'$ is the
universal profinite cover of $T'$ in view of
Proposition \ref{prop:separable} again. By adjunction, 
it follows that
\[ \Hom_{\Pro(\cC')}(G',H') \cong 
\Hom_{\cC'/\cF'}(T', Q(H')), \]
where $Q := Q^{\cC', \cF'}$.
This is a $\bQ$-vector space, since $T'$ is 
divisible; so we conclude as above.

The case where $G = \bG_a$ in characteristic $0$
is handled similarly: then  
\[ \Hom_{\Pro(\cC')}(G',H') \cong 
\Hom_{\cC'}(\bG'_a,H') \]
is again a $\bQ$-vector space, hence $\Gamma$-acyclic.

Next, let $G = \tW$ in characteristic $p > 0$. We obtain
as above
 \[ \Hom_{\Pro(\cC')}(G',H') \cong 
\Hom_{\Pro(\cC'/\cF')}(W',Q (H')), \]
where $Q := Q^{\cC',\cF'}$;
moreover, $W'$ is the projective cover of $\bG'_a$
in $\Pro(\cC'/\cF')$ in view of \cite[V.3.7.5]{DG}. 
To show that the above 
$\Gamma$-module is acyclic, we may assume that $H$ 
is simple in $\cC/\cF$ (since every object 
in that category has finite length, and the functor
$\Hom_{\Pro(\cC'/\cF')}(W', -)$ 
is exact). So $H$ is either a simple abelian variety, 
or a simple torus, or $\bG_a$ (see \cite[Prop.~3.2]{Br17a}). 
As $W$ is unipotent, we may further assume that 
$H = \bG_a$. We now need the following observation:

\begin{lemma}\label{lem:top}
Let $\cA$ be an abelian category, and $f : X \to Y$
an essential epimorphism in $\cA$, where $Y$ is
simple. For any simple object $Z$ of $\cA$, we have
$\Hom_{\cA}(X,Z) = 0$ unless $Z \cong Y$, and
$\Hom_{\cA}(X,Y) \cong \End_{\cA}(Y)$ via composition 
with $f$. 
\end{lemma}

\begin{proof}
Let $g \in \Hom_{\cA}(X,Z)$. If $g \neq 0$, then
the composition $\Ker(g) \to X \to Y$ is not an
epimorphism, since $f$ is essential. As $Y$ is simple,
this composition is zero, i.e., $\Ker(g) \subset \Ker(f)$.
This yields an exact sequence
\[ 0 \longrightarrow \Ker(f)/\Ker(g) \longrightarrow
X/\Ker(g) \longrightarrow Y \longrightarrow 0. \]
As $Z$ is simple, we have $X/\Ker(g) \cong Z$.
So $Z \cong Y$ and $\Ker(f) = \Ker(g)$, i.e.,
$g$ factors uniquely through $f$.
\end{proof}

Applying Lemma \ref{lem:top} to the abelian category 
$\Pro(\cC'/\cF')$ and to the essential epimorphism 
$W' \to \bG'_a$, we see that 
$\Hom_{\Pro(\cC'/\cF')}(W',H') = 0$
unless $H' \cong \bG'_a$, and
\[ \Hom_{\Pro(\cC'/\cF')}(W',\bG'_a) 
\cong \End_{\cC'/\cF'}(\bG'_a) \]
as $\Gamma$-modules. 

We now make a further reduction to the case where
$k$ \emph{is perfect}: indeed, the Galois group
$\Gamma$ is invariant under purely inseparable field
extensions of $k$, and the same holds for the isogeny 
category $\cC/\cF$ by \cite[Thm.~3.11]{Br17a}.
Recall that $\End_{\cC'}(\bG'_a)$ is the noncommutative
polynomial ring $k'[F]$, and
$\End_{\cC'/\cF'}(\bG'_a)$ is its fraction
skewfield $k'(F)$, as follows e.g. from
\cite[V.3.6.7]{DG}. To show that $k'(F)$ is 
acyclic, it suffices to check that it is the
direct limit of its $\Gamma$-submodules 
$g^{-1} k'[F]$ over all nonzero $g \in k[F]$,
since every such submodule is isomorphic
to $k'[F] \cong k' \otimes_k k[F]$, hence is
acyclic. For this, we adapt a standard argument
of commutative algebra.

Let $g^{-1}f \in k'(F)$, where $f,g \in k'[F]$
and $g \neq 0$. Since the left $k[F]$-module
$k'[F]$ is finitely generated and the ring $k[F]$ 
is left Noetherian, the increasing sequence of
submodules 
$k[F] + k[F]\, g + \cdots + k[F] \, g^n$ stops.
So there exist an integer $n \geq 1$
and $a_1,\ldots, a_n \in k[F]$ such that 
$g^n + a_1 g^{n-1} + \cdots + a_n = 0$.
Since $k'[F]$ is a domain and $g \neq 0$,
we may further assume that $a_n \neq 0$.
Then $g' g = - a_n \in k[F] \setminus \{ 0 \}$,
where 
$g' := g^{n-1} + a_1 g^{n-2} + \cdots + a_{n-1}$.
Thus, $g^{-1} f = (g'g)^{-1} g'f$ is as desired.

This completes the proof of the proposition for $G = \tW$, 
and leaves us with the case where $G$ is \emph{profinite}
(and $k$ is arbitrary). We now prove:

\begin{lemma}\label{lem:fin}
Let $G \in \Pro(\cF)$, $H \in \cC$, and 
$f \in \Hom_{\Pro(\cC')}(G',H')$. 
Then there exists a finite subgroup $F \subset H$
such that $f$ factors through $F' \subset H'$.
\end{lemma}

\begin{proof}
Write $G$ as a filtered inverse limit of finite quotients
$G_i$; then $G'$ is the filtered inverse limit of its
finite quotients $G'_i$. Thus,
\[ \Hom_{\Pro(\cC')}(G',H') = 
\lim_{\to} \Hom_{\cC'}(G'_i,H'). \]
We may therefore assume that $G \in \cF$; then
$\Im(f)$ is a finite $k'$-subgroup of $H'$. Let
$I \subset \Im(f)$ denote the largest infinitesimal
subgroup, then $I$ is contained in some Frobenius
kernel $\Ker(F^n_{H'/k'})$. Hence 
$I \subset \Ker(F^n_{H/k})' =: J'$,
where $J \subset H$ is infinitesimal. Thus,
$I = J' \cap \Im(f)$, and $\Im(f)/I$
is a finite \'etale $k'$-subgroup of
$H'/J' = (H/J)'$. So we may assume 
that $\Im(f)$ is \'etale; then we may view $\Im(f)$
as a finite subgroup of $H(k_s)$, stable under 
$\Gal(k_s/k')$. In that case, the (finitely many) 
conjugates of $\Im(f)$ under $\Gal(k_s/k)$ generate 
the desired finite $k$-subgroup $F \subset H$.
\end{proof}

By Lemma \ref{lem:fin}, we have
\[ \Hom_{\Pro(\cC')}(G',H') = 
\lim_{\to} \Hom_{\Pro(\cF')}(G',F'), \]
where the limit runs over all the finite subgroups 
$F \subset H$. Since taking $\Gamma$-cohomology 
commutes with direct limits, it suffices to show that 
\emph{the $\Gamma$-module 
$\Hom_{\Pro(\cF')}(G',H')$ 
is acyclic whenever $G$ is the projective 
cover of a finite simple group, and $H$ is finite}. 
We may further assume $H$ \emph{simple}.

Consider the Serre subcategory $\cI$ of $\cF$,
and recall that $\cF/\cI \cong \cE$.
By Remark \ref{rem:indproj},
the indecomposable projective objects of 
$\Pro(\cF)$ are exactly those of $\Pro(\cI)$ and the
universal pro-infinitesimal covers $\tP$, where 
$P \in \Pro(\cE)$ is indecomposable and projective.
Also, the universal pro-infinitesimal cover commutes 
with base change under separable algebraic field 
extensions by Remark \ref{rem:separable}.
As a consequence, we obtain
\[ \Hom_{\Pro(\cF')}(\tP', H') \cong
\Hom_{\Pro(\cE')}(P', Q(H')), \]
where $Q := Q^{\cF', \cI'}$.

To show that the above $\Gamma$-module is acyclic,
we may assume $H \in \cE$. We now adapt the
argument in the proof of \cite[Lem.~3.10]{Br17b}, by using
results of Galois cohomology from \cite[Chap.~II]{Se97}. 
Consider the Galois groups $\Gamma_k := \Gal(k_s/k)$ 
and $\Gamma_{k'} := \Gal(k_s/k')$;
these fit in an exact sequence
\[ 1 \longrightarrow \Gamma_{k'} \longrightarrow
\Gamma_k \longrightarrow \Gamma 
\longrightarrow 1. \]
By \cite[II.5.1.7]{DG}, $\cE$ is equivalent to 
the category $\Gamma_k-\mod$ of finite commutative groups
equipped with a discrete action of $\Gamma_k$. The 
latter category has a duality given by 
$M \mapsto \Hom(M,\bQ/\bZ)$, where the right-hand
side denotes the group homomorphisms on which 
$\Gamma_k$ acts via its given action on $M$ and
the trivial action on $\bQ/\bZ$. This yields an
anti-equivalence between $\cE$ and $\Gamma_k-\mod$,
which extends uniquely to an anti-equivalence
between $\Pro(\cE)$ and the category $\Gamma_k-\Mod$ 
of all discrete $\Gamma_k$-modules (the latter is the 
ind-category of $\Gamma_k-\mod$).
Under this anti-equivalence, the base change functor 
$\otimes_k \, k' : \Pro(\cE) \to \Pro(\cE')$
corresponds to the restriction from $\Gamma_k$
to $\Gamma_{k'}$. So it suffices to check that 
$\Hom^{\Gamma_{k'}}(M,N)$ is 
$\Gamma$-acyclic for any object $M$ of 
$\Gamma_k-\mod$ and any injective object $N$ 
of $\Gamma_k-\Mod$. 

We have an injective morphism of discrete 
$\Gamma_k$-modules
\[ \iota : N \longrightarrow
\Hom_{\cont}(\Gamma_k,N), \quad 
x \longmapsto (\gamma \longmapsto \gamma x),\]
where the right-hand side denotes
the group of continuous maps $\Gamma_k \to N$,
equipped with the action $\Gamma_k$ via right 
multiplication on itself. Since the $\Gamma_k$-module 
$N$ is injective, it is identified with a summand of 
$\Hom_{\cont}(\Gamma_k,N)$ via $\iota$; thus, 
the $\Gamma$-module $\Hom^{\Gamma_{k'}}(M,N)$
is a summand of
\[ \Hom^{\Gamma_{k'}}
(M, \Hom_{\cont}(\Gamma_k,N)) \cong 
\Hom^{\Gamma_{k'}}_{\cont}
(M \times \Gamma_k, N) \]
\[ \cong \Hom^{\Gamma_{k'}}_{\cont}
(\Gamma_k, \Hom(M,N)). \]
So it suffices in turn to show that the latter $\Gamma$-module
is acyclic. But since $P := \Hom(M,N)$ is a
discrete $\Gamma_k$-module, we have an isomorphism
\[ \Hom^{\Gamma_{k'}}_{\cont}(\Gamma_k,P) 
\stackrel{\cong}{\longrightarrow} \Hom(\Gamma,P) \]
that sends $f$ to the $\Gamma_{k'}$-invariant map
\[ \Gamma_k \longrightarrow P, \quad 
g \longmapsto g^{-1} f(g). \]
The inverse isomorphism sends 
$\varphi : \Gamma \to P$ to the map 
\[ \Gamma_k \longrightarrow P, \quad
g \longmapsto g \varphi(\bar{g}), \]
where $\bar{g}$ denotes the image of $g$ in
$\Gamma_k/\Gamma_{k'} = \Gamma_k$.
Moreover, $\Hom(\Gamma,P)$ is an acyclic 
$\Gamma$-module as desired. 

Thus, we may assume $G \in \Pro(\cI)$. Consider
the Serre subcategory $\cI_m$ of $\cI$; then
$\cI/\cI_m \cong \cI_u$. Using again 
Remarks \ref{rem:indproj} and \ref{rem:separable},
we are reduced to showing the above acyclicity
assertion, with $\cF$ replaced by $\cI_m$ or $\cI_u$.

By Cartier duality, $\cI_m$ is anti-equivalent to 
$\cE_p$ (see \cite[IV.1.3]{DG}); moreover, $\cE_p$
is self-dual via $\Hom(-, \bQ_p/\bZ_p)$. So the desired
assertion for $\cI_u$ follows from that for $\cE_p$.

Finally, if $G \in \Pro(\cI_u)$, then $G$ is the projective 
cover (in $\Pro(\cI_u)$ or equivalently in $\Pro(\cI)$, 
$\Pro(\cF)$, $\Pro(\cL)$, $\Pro(\cC)$) of the unique 
simple object, $\alpha_p$. Thus, $G'$ is the projective 
cover of $\alpha'_p$ in view of Lemma \ref{lem:cov}
below. Also, $H = \alpha_p$ and hence
$H' = \alpha'_p$. Using Lemma
\ref{lem:top}, it follows that 
\[ \Hom_{\Pro(\cI'_u)}(G', H') =
\End_{\cI'_u}(\alpha'_p) = k'. \]
Since the $\Gamma$-module $k'$ is acyclic, 
this completes the proof of Proposition \ref{prop:acyclic}, 
and hence of Theorem \ref{thm:milne}.

\begin{lemma}\label{lem:cov}
Let $G$ be the projective cover of $\alpha_p$, 
and $k'/k$ a finite separable field extension. 
Then $G'$ is the projective cover of $\alpha'_p$.
\end{lemma}

\begin{proof}
By Lemma \ref{lem:change}, $G'$ is projective in 
$\Pro(\cC')$. Also, $G \in \Pro(\cI_u)$ and hence
$G' \in \Pro(\cI_u')$ is projective there. Recall that
$\cI'_u$ has a unique simple object $\alpha'_p$,
and denote by $P'$ its projective cover. Then $G'$ 
is a direct product of copies of $P'$ in view of
\cite[V.2.4.6 b)]{DG}. Also, the natural map
\[ k' = \End_{\cC'}(\alpha'_p) \longrightarrow
\Hom_{\Pro(\cC')}(P', \alpha'_p) \]
is an isomorphism by Lemma \ref{lem:top}. So it suffices
to show that the analogous map 
\[ \varphi : k'  \longrightarrow
\Hom_{\Pro(\cC')}(G', \alpha'_p) \]
is an isomorphism as well. We have
\[ \Hom_{\Pro(\cC')}(G', \alpha'_p) =
\Hom_{\Pro(\cC)}(G, \R_{k'/k}(\alpha'_p)), \]
where the Weil restriction $\R_{k'/k}(\alpha'_p)$
is an iterated extension of $d := [k':k]$ copies of 
$\alpha_p$. Using Lemma \ref{lem:top} again, 
it follows that $\Hom_{\Pro(\cC')}(G', \alpha'_p)$
has dimension at most $d$ when viewed as a 
$k$-vector space. Since $\varphi$ is injective and
$k$-linear, and $k'$ has dimension $d$ when viewed
as a $k$-vector space, we conclude that $\varphi$
is an isomorphism. 
\end{proof}

\medskip

\noindent
{\bf Acknowledgments.} Many thanks to Cyril Demarche, 
Mathieu Florence, Roy Joshua, Bruno Kahn, Chu Gia Vuong 
Nguyen and Takeshi Saito for very helpful discussions on 
the topics of this paper. Special thanks to the referee
for a careful reading of the paper and valuable comments.

\bibliographystyle{amsalpha}

\end{document}